\documentclass[11pt]{article}

\usepackage[margin=1in]{geometry}
\usepackage{amsmath}
\usepackage{amsfonts}
\usepackage{amssymb}
\usepackage{amsthm}
\usepackage{mathtools}
\mathtoolsset{showonlyrefs}
\usepackage{dirtytalk}
\usepackage{tikz}
\usepackage{doi}
\usepackage{xcolor}
\usepackage{hyperref}
\usepackage[capitalise]{cleveref}

\numberwithin{equation}{section}

\newcommand{\D}[1]{\mathcal{D}(#1)}
\newcommand{\C}[1]{\mathcal{#1}}
\newcommand{\BR}{\mathbb{R}}
\newcommand{\B}[1]{{\bf #1}}

\newcommand{\ad}{{\rm{ad}}}
\newcommand{\mt}{{-i}}

\newcommand{\Lv}{L_{y_0}}

\DeclareMathOperator*{\am}{argmin}

\newcommand*\EX{\hspace*{0em plus 1fill}$\Box$}

\theoremstyle{plain}
\newtheorem{theorem}{Theorem}[section]

\newtheorem{proposition}[theorem]{Proposition}
\newtheorem{corollary}[theorem]{Corollary}
\newtheorem{definition}[theorem]{Definition}

\theoremstyle{definition}
\newtheorem{ass}{Assumption}[section]

\theoremstyle{remark}

\newtheorem{example}[theorem]{Example}

\newcounter{thmpart}[theorem]

  \newcounter{asspart}[ass]
  
  \newcommand{\asspartlabel}[1]{%
    \refstepcounter{asspart}\label{#1}\textbf{(\alph{asspart})}%
    }
    
    \crefname{theorem}{Theorem}{Theorems}
    \Crefname{theorem}{Theorem}{Theorems}
    \crefname{proposition}{Proposition}{Propositions}
    \Crefname{proposition}{Proposition}{Propositions}
    \crefname{lemma}{Lemma}{Lemmas}
    \Crefname{lemma}{Lemma}{Lemmas}
    \crefname{corollary}{Corollary}{Corollaries}
    \Crefname{corollary}{Corollary}{Corollaries}
    \crefname{definition}{Definition}{Definitions}
    \Crefname{definition}{Definition}{Definitions}
    \crefname{remark}{Remark}{Remarks}
    \Crefname{remark}{Remark}{Remarks}
    \crefname{ass}{Assumption}{Assumptions}
    \Crefname{ass}{Assumption}{Assumptions}
    \crefname{asspart}{Assumption}{Assumptions}
    \Crefname{asspart}{Assumption}{Assumptions}
    \crefname{thmpart}{Theorem}{Theorems}
    \Crefname{thmpart}{Theorem}{Theorems}
    \crefname{example}{Example}{Examples}
    \Crefname{example}{Example}{Examples}
    
    \newcommand{\TheTitle}{State constrained convex Nash equilibrium problems coupled with linear hyperbolic PDEs}
    \newcommand{\TheAuthors}{M. Bongarti and M. Hinterm\"uller}
    \newcommand{\email}[1]{\texttt{#1}}
    
    \title{\TheTitle}
    
    \author{Marcelo Bongarti\thanks{Weierstra\ss\ Institute for Applied
    Analysis and Stochastics, Berlin, Germany
    (\email{bongarti@wias-berlin.de}, corresponding author).}
    \and Michael Hinterm\"uller\thanks{Weierstra\ss\ Institute for
    Applied Analysis and Stochastics, Berlin, Germany, and
    Humboldt-Universit\"at zu Berlin, Berlin, Germany.}}
    
    \date{}
    
    \hypersetup{
      pdftitle={\TheTitle},
        pdfauthor={\TheAuthors}
        }
        
\begin{document}
        
        \maketitle
        
        \begin{abstract}
        \noindent We study the existence of equilibria for state constrained, convex generalized Nash equilibrium problems (GNEPs) coupled with hyperbolic partial differential equations (PDEs). Analogous problems have been addressed for state constrained GNEPs coupled with elliptic and parabolic PDEs, respectively, but the problematic regularity of the set-valued constraint maps has been a barrier for development of an existence theory in the hyperbolic case. This is mainly due to compactness issues with the strategy-to-state maps. Beyond existence, we also provide first-order optimality conditions for a class of fairly general linear hyperbolic PDEs and show its relevance for applications such as the wave equation, advertising dynamics, and the linearized isothermal Euler system on networks.
        \end{abstract}
        
        
        
        \bigskip
        
        \section{Introduction}

Noncooperative games with PDE constraints naturally arise in applications where multiple agents (or players) influence a shared physical system; see, e.g., \cite{hintermuller_generalized_2015,hintermueller_pde-constrained_2013,roubicek_nash_2007,roubicek_noncooperative_1999,lasiecka_control_2000}. Specific examples involving hyperbolic PDEs include gas transport, vibration control, traffic flow, and advertising models governed by transport equations \cite{bongarti_optimal_2024,machowska_closed-loop_2022,dvurechensky_cournot-nash_2024}. 

In this paper we consider $N$-player games, with $N>1$. 
Each player $i\in\{1,\dots,N\}$ controls a decision variable $u_i$ by solving an optimization problem of the general form
\begin{equation}\label{gen_game}
    \min_{u_i} \mathcal{J}_i(u_i),
\end{equation}
where $\mathcal{J}_i$ represents the player's objective. We keep the notation in \eqref{gen_game} simple, but the modeling flexibility (and hence complexity) of such games stems from the wide range of possible constraints and interaction mechanisms that can be present in \eqref{gen_game}. 
In particular, the decision problem of each player may depend on the strategies of the others, both through the objective functional and through the feasible set.

In view of the latter, we analyze a class of $N$-player games featuring three distinct types of coupling: (i) linear \emph{hyperbolic} PDE coupling; (ii) objective coupling; and (iii) geometric (feasibility) coupling. Objective coupling, where $\mathcal{J}_i$ may depend on the full strategy profile $u=(u_1,\dots,u_N)$, is standard in noncooperative game theory. 
In contrast, PDE coupling and geometric feasibility coupling are significantly less explored and, 
to the best of the authors' knowledge, the \textit{simultaneous} presence of hyperbolic PDE coupling and geometric feasibility constraints has not yet been addressed in the existing literature.

Motivated by the aforementioned games involving PDEs, we consider dynamics given by the following family of abstract Cauchy problems: 
\begin{equation}\label{intro_state}
\begin{cases}
y_t(t) = \mathcal{A}_iy(t) + \mathcal{P}_i(t)y(t) + (\mathcal{B}_iu)(t) + f_i(t),
& t>0,\\[1mm]
y(0)=y_0^i,
\end{cases}
\end{equation}
posed in a Banach space $Y$. 
Here, $\mathcal{A}_i$ generates a $C_0$-semigroup on $Y$ and the linear operators $\mathcal{P}_i(\cdot)$ and $\mathcal{B}_i$ satisfy assumptions ensuring well-posedness of \eqref{intro_state} in an appropriate sense, specified later. The family is coupled because the solution of \eqref{intro_state}, for each $i$, depends on the full strategy vector $u = (u_1, \cdots, u_N)$. 

Given a finite time horizon $T>0$, the dynamics \eqref{intro_state} induces the family of \emph{strategy-to-state} maps $u \mapsto \mathcal{S}_i(u): [0,T] \mapsto Y$, which enter the objective coupling. A fairly typical structure is $\mathcal{J}_i(u) = \mathcal{J}_i^1(\mathcal{S}_i(u)) + \mathcal{J}_i^2(u_i)$ so that each objective depends on the entire strategy bundle through the state. This already introduces a global interaction mechanism.

The most intricate coupling arises at the feasibility level. In standard Nash equilibrium problems, each player has a fixed individual constraint set. Here we allow the feasible set of player $i$ to depend on the competitors' strategies. Denoting by $u_\mt$ the vector collecting the $N-1$ strategies of all players except $i$, we write the feasible set as $\B{U}_i(u_\mt)$. Each player then solves
\begin{equation}\label{intro_gen_game}
    \min\limits_{v_i \in \B{U}_i(u_\mt)} \C{J}_i(v_i,u_\mt)
\end{equation}
which leads us to a \emph{generalized} Nash equilibrium problem (GNEP). 

Under well-posedness of problem \eqref{intro_gen_game}, one can define the best-response set-valued map 
\begin{equation}\label{intro_best}
    u \mapsto \C{V}(u), \qquad \C{V}(u) = \left\{v = (v_1, \cdots, v_N) : v_i \in \am_{u_i \in \B{U}_i(u_\mt)} \C{J}_i(v_i, u_\mt)\right\}.
\end{equation} 
A generalized Nash equilibrium (GNE) is a fixed point of this map, i.e., $\overline{u}$ is a GNE iff $\overline{u} \in \C{V}(\overline{u})$. 

The central analytic difficulty in proving existence of generalized Nash equilibria lies in the regularity properties of the set-valued map $u_\mt \mapsto \B{U}_i(u_\mt)$. Classical existence results -- such as \cite{arrow_existence_1954} -- rely on continuity of $\B{U}_i(\cdot)$. However, for set-valued maps, continuity is understood as the combination of upper and lower semicontinuity properties. While such properties are more likely to hold when $\B{U}_i(u_\mt)$ depends on $u_\mt$ algebraically, the situation is substantially more involved when the feasible sets have the form
\begin{equation}
    \label{intro_feasible} \B{U}_i(u_\mt) = \{u_i : \C{S}_i(u) \in K_i\}, \qquad K_i \ \text{compact,}
\end{equation}
which is the case of interest in this work. 

In \eqref{intro_feasible}, the coupling acts through the solution operator of the hyperbolic PDE \eqref{intro_state}. The feasibility condition $\C{S}_i(u) \in K_i$ requires sufficient regularity of the state $y_i = \C{S}_i(u)$ in order to be even well-defined. Moreover, regularity properties of the set-valued map \eqref{intro_feasible} depend on structural properties of the strategy-to-state map itself. Practically, this establishes a coupling between the players' decisions, and mathematically an interplay between PDE and equilibrium theories.

Since differential operators typically enjoy favorable closedness properties, upper semicontinuity of maps as in \eqref{intro_feasible} can often be obtained under mild assumptions. The cumbersome property is lower-semicontinuity. In fact, in the few works where feasible sets of the form \eqref{intro_feasible} appear, either the game is jointly convex so that the family of optimization problems can be reduced to a single problem via the Nikaido-Isoda function \cite{nikaido_note_1955}, as in \cite{hintermuller_generalized_2015,dvurechensky2024cournotnashmodelcoupledhydrogen}, or extra regularity assumptions are imposed. For instance, the notion of \emph{strict uniformity of feasible responses} in \cite{hintermueller_pde-constrained_2013} allows one to obtain a generalized Nash equilibrium as a uniform limit of Nash equilibria.

This assumption reflects a structural difficulty. Indeed, verifying lower semicontinuity of maps of the form \eqref{intro_feasible} is challenging even for elliptic and parabolic PDEs, which possess stronger regularity properties than hyperbolic equations. The reason is that lower semicontinuity depends not only on analytic features such as compactness and closedness, but also on geometric properties such as interiority of feasible points. These issues are particularly delicate in infinite dimensional and function space settings.

In this paper we provide a new proof of existence of equilibria for convex generalized Nash games, with specific emphasis on PDE-constrained problems, although the approach applies more broadly. The method builds on the framework recently developed in \cite{bongarti2025structureversusregularitysetvalued}. The central idea is to replace regularity requirements on the constraint maps by structural conditions. In analogy with jointly convex case, existence can be obtained when the bundle constraint map satisfies suitable convexity-type properties. In \cite{bongarti2025structureversusregularitysetvalued} we covered two classes: graph-convex and KKM ( Knaster-Kuratowski-Mazurkiewicz) maps. Both structures allow us to avoid lower-semicontinuity assumptions on $\B{U}_i(\cdot).$ 

In finite-dimensional game theory, the advantage of using graph-convexity is mainly procedural, since graph-convexity of a map is known to be equivalent to lower semicontinuity on the interior of its domain. In infinite-dimensional settings, however, graph-convexity becomes substantially more powerful. It permits the treatment of games whose constraint maps have domains with an empty interior, a quite frequent situation in function spaces. Moreover, graph-convexity is easier to verify than lower-semicontinuity, even in simplified PDE-constrained models.

In practical applications, the coupling mechanisms discussed above arise in concrete structural forms. In energy systems, for example, pressure bounds and capacity restrictions are imposed on gas or hydrogen transport networks. These restrictions lead naturally to state constraints in equilibrium models that combine market interaction with physical flow dynamics. While hyperbolic equilibrium problems without state constraints have been studied in the literature, see, e.g., \cite{lions_hierarchic_1994,lions_remarks_1994} and references therein), the case with state constraints is in essence void. 

A fundamental structural distinction between parabolic and hyperbolic dynamics plays a decisive role in this context. Parabolic operators have smoothing effects. As a consequence, the associated strategy-to-state maps are often compact in natural functional spaces. This compactness can be exploited in equilibrium analysis, since it supports sequential closedness arguments that underlie classical fixed-point approaches and Kakutani-type theorems \cite{arrow_existence_1954,bensoussan_impulse_1987,chan_generalized_1982}.

In contrast, evolution families generated by hyperbolic operators do not regularize, in general. Consequently, the strategy-to-state map for a hyperbolic PDE is usually not compact in most natural function spaces. Thus, a general existence theory for convex state-constrained GNEPs coupled by hyperbolic PDEs has remained largely underdeveloped. 

We believe this paper makes a valuable contribution to the topic and illustrates the reach of the framework. The abstract assumptions are verified for three representative hyperbolic models: (i) the wave equation with Dirichlet boundary conditions; (ii) a first-order advertising--goodwill model \cite{machowska_closed-loop_2022}; and (iii) the linearized isothermal Euler system on tree-shaped gas networks \cite{bongarti_optimal_2024}. 
These examples demonstrate that the proposed approach applies to relevant classes of PDE-constrained games arising in mechanics, transport, and energy systems.

\medskip

\noindent{\bf Outline.} \Cref{sec:existence_regularity} sets up the abstract state equation, lists the standing assumptions on the generator, the perturbation, and the controls, and establishes well-posedness and basic regularity of the strategy-to-state maps $\C{S}_i$, in both a classical and a dual classical setting. \Cref{sec:existence_equilibria} contains our main existence result: we show that graph-convexity of the coupling constraint map, together with weak upper semicontinuity, is enough to guarantee a generalized Nash equilibrium, with no lower-semicontinuity assumption on the feasible sets; we also explain how to choose the ambient space $W$ so that the required compactness follows from an Aubin--Lions argument. \Cref{sec:two_examples} illustrates these abstract hypotheses on two model problems, the controlled wave equation and a first-order advertising-goodwill equation.

\Cref{sec:characterization} turns the existence theory into a characterization of equilibria through first-order optimality conditions, obtained via an abstract Lagrange multiplier rule together with a suitable notion of continuity point. Finally, \Cref{ex_gas} applies the full machinery to a noncooperative game on a small gas network governed by the linearized isothermal Euler equations, in which a network operator and two competing suppliers act through boundary controls; we prove existence of an equilibrium and derive its first-order optimality system explicitly, including the adjoint equations satisfied by each player.

\medskip

\noindent{\bf Notation.} All vector spaces are real Banach spaces. For a Banach space $X$, its dual is $X^*$. The norm in $X$ is denoted by $\|\cdot\|_X$ and when there is no risk of confusion we write $\|\cdot\|$. For an interval $(0,T)$, $T>0$, and a Banach space $X$,
  $
    L^p(0,T;X),\; 1 \le p \le \infty,
  $
  denotes the usual Bochner spaces \cite{MR2500068}, and
  $
    C([0,T];X)
  $
  is the space of continuous functions from $[0,T]$ into $X$.
  We use $H^1(0,T;X)$ for the standard Sobolev space of $X$–valued functions
  with square-integrable weak time derivative.

  If $E$ is a Banach space, then
  $
    BV((0,T);E^*)
  $
  denotes the space of functions of bounded variation from $[0,T]$ into $E^*$ \cite{MR1079985}. The space of $E^*$–valued (finite) Radon measures on $(0,T)$ is written as
  $
    \C{M}(0,T;E^*) \coloneqq C([0,T];E)^*.
  $ For Banach spaces $X,Y$, the product space is $X\times Y$ with norm
  $
    \|(x,y)\|_{X\times Y} \coloneqq \|x\|_X + \|y\|_Y.
  $
  Finite Cartesian products are treated analogously (e.g., $U=\prod_{i=1}^N U_i$).

  If $X$ embeds continuously into $Y$ we write
  $
    X \hookrightarrow Y.
  $
  If the embedding is compact, we write
  $
    X \xhookrightarrow{c} Y.$ Convergence in norm (strong convergence) of a sequence $(x_n)$ in a Banach space $X$ is denoted by
  $
    x_n \to x \quad\text{in } X.
  $
  Weak convergence in $X$ is denoted by
  $
    x_n \rightharpoonup x \quad\text{in } X.
  $
  Weak-* convergence in a dual space $X^*$ is denoted by
  $
    x_n^* \rightharpoonup^* x^* \quad\text{in }X^*.
  $ When needed, we indicate the space explicitly, e.g.
  $
    u_n \rightharpoonup u \quad\text{in } L^2(0,T;Y)$ or $
    y_n \to y \quad\text{in } C([0,T];E).
  $

  Finally, we use the notation $F: X \rightrightarrows Y$ when $F$ is a set-valued map, i.e., for each $x$ in the domain of $F$, $F(x)$ is a subset of $Y$. The graph of a set-valued map is given by ${\rm Gr}(F) = \{(x,y):y \in F(x)\}.$

\section{Existence and basic regularity of the states}\label{sec:existence_regularity}

In this section we discuss basic well-definedness and regularity of the map $u \mapsto \C{S}_i(u)$ for each $i = 1, \cdots, N$ in the following setting.
For this purpose, let $Y$ and $U$ be Banach spaces such that 
$U = U_1\times\ldots\times U_N$, for reflexive Banach spaces $U_i$ for all $i = 1,\dots,N$.

We consider the following abstract Cauchy problem as the underlying state equation:
\begin{equation}\label{abs1_dis_in}
  \begin{cases}
    y_t(t) = \C{A}y(t) + \C{P}_i(t)y(t) + (\C{B}_i u)(t) + f(t), & t > 0,\\[1mm]
    y(0) = y_0.
  \end{cases}
\end{equation}
The various (given) constituents in \eqref{abs1_dis_in} satisfy the following structural conditions.

\begin{ass}\label{ass_a}
\ 
\begin{itemize}
  \item[\asspartlabel{ass_a-a}]
    $\C{A} : \D{\C{A}} \subset Y \to Y$ is a closed, densely defined linear operator that generates a strongly continuous (or $C_0$) semigroup $\{S(t)\}_{t \geqslant 0}$.
   \item[\asspartlabel{ass_a-b}]
    For each $i = 1,\dots,N$, the linear operator
    $\C{B}_i : U \to L^2(0,T;Y)$ is weak-to-strong continuous.
  \item[\asspartlabel{ass_a-c}]
    For each $i = 1,\dots,N$ and every $t \in [0,T]$, the linear operator
    $\C{P}_i(t) : Y \to Y$ is bounded, and for every $y \in Y$ we have
    $\C{P}_i(\cdot)y \in C([0,T];Y)$.
    \item[\asspartlabel{ass_a-e}] $f \in L^1(0,T;Y)$, and $y_0\in Y$.
\end{itemize}
\end{ass}

Problem \eqref{abs1_dis_in} together with \Cref{ass_a} is the prototype of the linear hyperbolic state equation studied in this paper.  
Our arguments can be extended to more general linear hyperbolic models with minor (mostly notational) changes, but we keep this structure for clarity.

\subsection{Existence of the strategy-to-state map}\label{state_prop}

Let $\{S(t)\}_{t \geqslant 0}$ be the semigroup generated by the operator $\C{A}$; see \Cref{ass_a-a}.  
Then there exist constants $M \geqslant 1$ and $\omega \geqslant 0$ such that
\begin{equation}
  \label{exp_est}
  \|S(t)y\|_{Y} \leqslant M e^{\omega t} \|y\|_Y,
\end{equation}
for all $t \geqslant 0$ and $y \in Y$; see \cite{pazy_semigroups_1983}.  
Since this estimate holds for any $C_0$–semigroup, we reuse the same constants $M$ and $\omega$ whenever convenient, even if the underlying semigroup changes.

By the uniform boundedness principle and \Cref{ass_a-c}, we obtain
\begin{equation}
  \label{unif_P}
  P_i \coloneqq \max_{t \in [0,T]} \|\C{P}_i(t)\|_{\C{L}(Y)} < +\infty.
\end{equation}

For fixed $y_0$, we formally define the operator $u \mapsto \Lv u$ implicitly by the variation-of-parameters formula
\begin{equation}
  \label{vop_formula}
  [\Lv u](t)
  = S(t)y_0
    + \int_0^t S(t-\tau)\bigl[\C{P}_i(\tau)[\Lv u](\tau) + (\C{B}_i u)(\tau)+f(\tau)\bigr]\,d\tau,
\end{equation}
for $t\in (0,T]$. Grönwall’s inequality implies that \eqref{vop_formula} has at most one solution.  
Hence $\Lv$ is single-valued: if a solution exists for a given $u$, then it is unique.




\begin{definition}
  A function $y: [0,T] \to Y$ is a \textbf{classical} solution of \eqref{abs1_dis_in} if
  \begin{itemize}
    \item[\bf (i)] $y$ is continuously differentiable as a $Y$–valued function;
    \item[\bf (ii)] $y(t) \in \C{D}(\C{A})$ for all $t \in [0,T]$ and $y$ is continuous as a $\C{D}(\C{A})$–valued function;
    \item[\bf (iii)] $y(0) = y_0$ and $y$ satisfies \eqref{abs1_dis_in} for all $t \in (0,T]$.
  \end{itemize}
\end{definition}

In many applications, in particular games under pointwise state constraints, one needs such regular solutions.
The following result is an adaptation of standard results; see, e.g., \cite{bensoussan_impulse_1987}, and needs extra regularity of the data as invoked next.

\begin{ass}\label{classical_setting}
\ 
\begin{enumerate}
    \item[\asspartlabel{classical_setting-a}] $y_0 \in \C{D}(\C{A})$ and $u \in U$;
    \item[\asspartlabel{classical_setting-b}] $\C{P}_i(\cdot)y \in C^1([0,T];Y)$ for each $y \in Y$; 
    \item[\asspartlabel{classical_setting-c}] 
      $\C{B}_i \in \C{L}(U;H^1(0,T;Y))$; and
      \item[\asspartlabel{classical_setting-d}]
      $f \in H^1(0,T;Y).$
\end{enumerate}
\end{ass}

\begin{theorem}
  \label{wpp_classic}
  Under \Cref{classical_setting}, there exists a unique classical solution of \eqref{abs1_dis_in}.  
  Its regularity space is
  \begin{equation}
    \label{X_T}
    X_T \coloneqq C^1([0,T];Y) \cap C([0,T];\C{D}(\C{A})).
  \end{equation}
\end{theorem}

In particular, any classical solution admits the representation \eqref{vop_formula}.

\begin{corollary}
  If a classical solution of \eqref{abs1_dis_in} exists, then it is given by $y = \Lv u$ in \eqref{vop_formula}.
\end{corollary}

We next relax the regularity requirements on the data.  
Observe that the integral equation \eqref{vop_formula} is well-defined without assuming $y_0 \in \C{D}(\C{A})$, $\C{B}_i \in \C{L}(U;H^1(0,T;Y))$ or $\C{P}_i(\cdot)y \in C^1([0,T];Y)$.  
\Cref{ass_a} is sufficient.  
This motivates the following notion.

\begin{definition}
  Let $y_0 \in Y$ and $u \in U$.  
  A function $y = \Lv u \in C([0,T];Y)$ that solves \eqref{vop_formula} is called a \textbf{mild} solution of \eqref{abs1_dis_in}.
\end{definition}

Mild solutions do not exist in full generality, but in our setting they are well-posed.

\begin{theorem}
  \label{wpp_mild_th}
  Let $y_0 \in Y$ and $u \in U$.  
  Then \eqref{vop_formula} has a unique solution $y = \Lv u$ in $C([0,T];Y)$.  
  In particular, \eqref{abs1_dis_in} has a unique mild solution.
\end{theorem}

Sometimes the data are even less regular and live naturally in dual spaces.  
We then use duality to extend $\C{A} : \C{D}(\C{A}) \to Y$ to an operator
\[
  \overline{\C{A}} : Y \to \C{D}(\C{A}^*)',
\]
which generates a semigroup on $\C{D}(\C{A}^*)'$ extending $\{S(t)\}_{t \geqslant 0}$.  
By a standard abuse of notation, we again denote this extension by $\C{A}$ and its semigroup by $\{S(t)\}_{t \geqslant 0}$. In connection with this setting, we invoke the following assumption.

\begin{ass}\label{dual_setting}
\ 
\begin{enumerate}
    \item[\asspartlabel{dual_setting-a}] $y_0 \in \C{D}(\C{A}^*)'$ and $u \in U$;
    \item[\asspartlabel{dual_setting-b}] $\C{B}_i \in \C{L}(U;L^2(0,T;\C{D}(\C{A}^*)'))$; 
    \item[\asspartlabel{dual_setting-c}] $\C{P}_i(\cdot) : \C{D}(\C{A}^*)' \to \C{D}(\C{A}^*)'$ is bounded and
      $\C{P}_i(\cdot)y \in C([0,T];\C{D}(\C{A}^*)')$ for every $y$; and
      \item[\asspartlabel{dual_setting-d}]
      $f \in H^1(0,T;\C{D}(\C{A}^*)')$.
  \end{enumerate}
\end{ass}
\begin{theorem}
  \label{wpp_dual}
 Under \Cref{dual_setting}, \eqref{abs1_dis_in} has a unique solution  $$y \in C([0,T];\C{D}(\C{A}^*)')$$ with the same abstract representation as in the mild and classical cases.
\end{theorem}

\subsection{Basics of the strategy-to-state map}\label{section_basic_S}

We now turn to continuity and differentiability properties of the strategy-to-state map $u \mapsto \C{S}_i(u)$. For each $j = 1,\dots,N$, define the extension operator $e_j(u_j) \in U$ by
\begin{equation}\label{h_ext}
  e_j(u_j) = (\hat u_k)_{k=1}^N,
  \qquad
  \hat u_k =
  \begin{cases}
    u_j, & k = j, \\
    0,   & k \neq j,
  \end{cases}
\end{equation}
and set $\C{B}_i^j \coloneqq \C{B}_i \circ e_j : U_j \to L^2(0,T;Y)$.  
Then $e_j(u_j) \in U$ for every $u_j \in U_j$ and $\|\C{B}_i^j\| \leqslant \|\C{B}_i\|$ for all $j$.  
By linearity of $\C{B}_i$, any $u = (u_1,\dots,u_N) \in U$ satisfies
\begin{align}
  \label{B_char}
  \C{B}_i u
  = \C{B}_i\Bigl(\sum_{j=1}^N e_j(u_j)\Bigr)
  = \sum_{j=1}^N \C{B}_i e_j(u_j)
  = \sum_{j=1}^N \C{B}_i^j u_j.
\end{align}

This decomposition yields a convenient splitting of the state operator $u \mapsto \C{S}_i(u)$:
\begin{equation}
  \label{dec}
  \C{S}_i(u)
  = \C{S}_i(0) + \sum_{j = 1}^N \C{S}_i^0(e_j(u_j)),
\end{equation}
where $\C{S}_i^0(e_j(u_j)) = y_j$ solves
\begin{equation}
  \label{abs1_dk}
  \begin{cases}
    y_t(t) = \C{A}y(t) + \C{P}_i(t)y(t) + (\C{B}_i^j u_j)(t)+f(t), & t > 0,\\
    y(0)   = 0,
  \end{cases}
\end{equation}
and $\C{S}_i(0) = y^0$ solves
\begin{equation}
  \label{abs1_dk0}
  \begin{cases}
    y_t(t) = \C{A}y(t) + \C{P}_i(t)y(t)+f(t), & t > 0,\\
    y(0)   = y_0.
  \end{cases}
\end{equation}

By \Cref{wpp_mild_th}, each operator $\C{S}_i^{0,j} \coloneqq \C{S}_i^0 \circ e_j$ is well-defined, linear and bounded.  
It is convenient to write the integral representation of $\C{S}_i(u)$ as
\begin{align}
  \label{sol_formula}
  \C{S}_i(u)(t)
  &= S(t)y_0
     + \int_0^t S(t-\tau)\left[\C{P}_i(\tau)\C{S}_i(0)(\tau) + f(\tau)\right]\,d\tau \nonumber \\
  &\quad + \sum_{j=1}^N \int_0^t S(t-\tau)\bigl[\C{P}_i(\tau)
      \C{S}_i^{0,j}(u_j)(\tau) + (\C{B}_i^j u_j)(\tau)\bigr]\,d\tau.
\end{align}

From \eqref{sol_formula} we obtain the following Lipschitz estimates.

\begin{proposition}
  \label[proposition]{lip_cont}
  Under the assumptions of \Cref{wpp_classic,wpp_mild_th,wpp_dual}, the solution map
  $u \mapsto \C{S}_i(u)$ is Lipschitz continuous:
  \begin{itemize}
    \item[\bf (i)] from $U$ to $X_T$ if $y_0 \in \C{D}(\C{A})$ (classical setting);
    \item[\bf (ii)] $U$ to $C([0,T];Y)$ if $y_0 \in Y$ (mild setting);
    \item[\bf (iii)] from $U$ to $C([0,T];\C{D}(\C{A}^*)')$ if $y_0 \in \C{D}(\C{A}^*)'$ (dual setting).
  \end{itemize}
\end{proposition}
\begin{proof}
    Notice that for $u, v \in U$, the difference $\Lv u - \Lv v$ solves the following implicit integral equation
    \begin{equation}
        [\Lv u - \Lv v](t) = \int_0^t S(t-\tau)\left\{\C{P}_i(\tau)\left[\Lv u - \Lv v\right](\tau) + \C{B}_i(u-v)(\tau)\right\}d\tau
    \end{equation}
    whereby
    \begin{align*}
        &\left\|[\Lv u - \Lv v](t)\right\|_{Y} \nonumber \\ 
        &\leqslant \int_0^t Me^{\omega(t-\tau)}\left\{P_i\left\|\left[\Lv u - \Lv v\right](\tau)\right\|_Y + \left\|\C{B}_i(u-v)(\tau)\right\|_Y\right\}d\tau \nonumber \\
        &\leqslant MP_i\int_0^t e^{\omega(t-\tau)}\left\|\left[\Lv u - \Lv v\right](\tau)\right\|_Y d\tau + Me^{\omega T}\int_0^t \left\|\C{B}_i(u-v)(\tau)\right\|_{\C{D}(\C{A})'}d\tau \nonumber \\
        &\leqslant MP_i\int_0^t e^{\omega(t-\tau)}\left\|\left[\Lv u - \Lv v\right](\tau)\right\|_Y d\tau + Me^{\omega T}\sqrt{T}\|\C{B}_i(u-v)\|_{L^2(0,T;Y)} \nonumber \\
        &\leqslant MP_i\int_0^t e^{\omega(t-\tau)}\left\|\left[\Lv u - \Lv v\right](\tau)\right\|_Y d\tau + Me^{\omega T}\sqrt{T}\|\C{B}_i\|\|u-v\|_U.
    \end{align*}
    Grönwall's inequality then gives
    \begin{align*}
        &\left\|[\Lv u - \Lv v](t)\right\|_Y \leqslant M^2P_i\|\C{B}_i\|\sqrt{T}\dfrac{e^{\omega T}}{\omega}\left(e^{\omega T}-1\right)\|u-v\|_U,
    \end{align*}
    and this finishes the proof for the case (ii). The case (iii) and the estimate in $C([0,T],\C{D}(\C{A}))$ for the case (i) follow a very similar strategy. 
    
    To finish the proof, we need an estimate for $\|\partial_t (\Lv u - \Lv v)(t)\|_{Y}.$ To that end, we notice that differentiating \eqref{abs1_dis_in} in time and making $z = y_t$ we have
    \begin{equation}\label{abs1_dis_in_dif} 
    \begin{cases}
        z_t(t) = \C{A}z(t) + \partial_t\C{P}_i(t)y(t) + \C{P}_i(t)z(t) + \partial_t(\C{B}_iu)(t) + \partial_t f(t), \qquad t > 0 \\
        z(0) = y_t(0) = \C{A}y_0 + \C{P}_i(0)y_0 +(\C{B}_i u)(0) + f(0) \in Y.
    \end{cases}
\end{equation} 
By using  (i) and (ii) in \Cref{wpp_classic}, we can estimate $\|\partial_t (\Lv u - \Lv v)(t)\|_{Y}$ just like we estimated $\|(\Lv u - \Lv v)(t)\|_{Y}$ using \eqref{abs1_dis_in_dif} in place of \eqref{abs1_dis_in}.
\end{proof}

The next result concerns differentiability.

\begin{proposition}\label[proposition]{diff_sm}
  Let $\C{O} \subset U$ be open (for any norm on $U$).  
  Then $\C{S}_i$ is continuously Fréchet differentiable on $\C{O}$ with derivative $\C{S}_i' = \C{S}_i^0$.  
  More precisely:
  \begin{itemize}
    \item[\bf (i)] $\C{S}_i : \C{O} \to X_T$ is $C^1$ if $y_0 \in \C{D}(\C{A})$;
    \item[\bf (ii)] $\C{S}_i : \C{O} \to L^2(0,T;Y)$ is $C^1$ if $y_0 \in Y$;
    \item[\bf (iii)] $\C{S}_i : \C{O} \to L^2(0,T;\C{D}(\C{A}))$ is $C^1$ if $y_0 \in \C{D}(\C{A})'$.
  \end{itemize}
\end{proposition}
\begin{proof}
    Let $u \in \C{O}$ and $h \in U$ such that $u+h \in \C{O}$. From the variation-of-parameters formula \eqref{vop_formula} we have, for all $t \in [0,T]$,
  \begin{equation}
    \label{diff_0}
    [\Lv(u+h) - \Lv u](t)
    = \int_0^t S(t-\tau)\Bigl\{\C{P}_i(\tau)\bigl[\Lv(u+h) - \Lv u\bigr](\tau)
      + \bigl(\C{B}_i h\bigr)(\tau)\Bigr\}\,d\tau.
  \end{equation}

  Next, we know that $\C{S}_i^0(h) = L_0h$. By defining 
  $$z_h(t) \coloneqq \bigl[\Lv(u+h) - \Lv u\bigr](t) - [L_0h](t),
    \qquad t \in [0,T]$$
  we have
  $$z_h(t) = \int_0^t S(t-\tau)\,\C{P}_i(\tau)\,z_h(\tau)\,d\tau,
    \qquad t \in [0,T]$$
    and, in particular, $z_h(0)=0.$

  By the same Grönwall-type argument used in the proof of \Cref{lip_cont} (applied with right-hand side equal to zero), the only solution of this homogeneous integral equation in the relevant space (either $X_T$, $C([0,T];Y)$, or $C([0,T];\C{D}(\C{A}^*)')$) is $z_h \equiv 0$.  
  Therefore
 $$\Lv(u+h) - \Lv u = L_0h,    \qquad\text{for all } h \text{ with } u+h \in \C{O}.$$
  Equivalently,
 $$\C{S}_i(u+h) - \C{S}_i(u) - L_0h = 0.$$

  Hence, for any of the norms considered in (i)–(iii),
  $$\frac{\|\C{S}_i(u+h) - \C{S}_i(u) - \C{S}_i^0(h)\|}{\|h\|_U} = 0,
    \qquad\text{whenever } h \neq 0 \text{ and } u+h \in \C{O},$$
  which shows that $\C{S}_i$ is Fréchet differentiable at $u$ with derivative
  $$\C{S}_i'(u)h = \C{S}_i^0(h),
    \qquad h \in U.$$

  The operator $\C{S}_i^0$ is linear by construction and bounded as a map
  $$
    \C{S}_i^0 : U \to X_T,\quad
    \C{S}_i^0 : U \to L^2(0,T;Y),\quad
    \C{S}_i^0 : U \to L^2(0,T;\C{D}(\C{A})),
  $$
  respectively, by the well-posedness results \Cref{wpp_classic,wpp_mild_th,wpp_dual} and the Lipschitz estimates in \Cref{lip_cont}.  
  In particular, the derivative does not depend on $u$, so $\C{S}_i'$ is constant on $\C{O}$ and therefore continuous in each of the three settings.

  This proves that $\C{S}_i$ is $C^1$ on $\C{O}$, with $\C{S}_i'(u) = \C{S}_i^0$ for all $u \in \C{O}$, in each of the three cases (i)–(iii).
\end{proof}

\section{Existence of equilibria for state constrained games}\label{sec:existence_equilibria}

In this section we investigate existence of equilibria for convex hyperbolic constrained Nash games. We consider a game with $N > 1$ players (or agents).
The strategy set of player $i$, $i\in\{1,\ldots,N\}$, is a reflexive Banach space $U_i$, and the player's private (or personal) constraint set is a nonempty, closed and convex subset $U_i^\ad \subset U_i$.  
We define the joint strategy space and the set of feasible strategies as $U = U_1\times\ldots \times U_N$ and $U^\ad = U_1^\ad\times\ldots \times U_N^\ad$, respectively.
For each $i$, we denote by $u_{-i}$ the projection of a bundle $u \in U$ onto $U_{-i} \coloneqq \prod_{j \neq i} U_j$,
and we write $u = (u_i,u_{-i})$ to highlight the $i$th component, without changing the order of the coordinates of $u$.

A generalized Nash equilibrium problem (GNEP) is a collection of $N$ coupled optimization problems of the form
\begin{equation}\label{gamei}\tag{$P_i$}
  \begin{cases}
    \text{Given } u_{-i} \in U_{-i}^\ad,\\[1mm]
    \text{Minimize } \C{J}_i(u_i,u_{-i})
    \text{ subject to (s.t.) } u_i \in U_i^\ad \cap \B{U}_i(u_{-i}).
  \end{cases}
\end{equation}
Here, $\C{J} = \{\C{J}_i\}_{i=1}^N$ is a family of objective functionals, and $\{\B{U}_i\}_{i=1}^N$ is a family of coupling constraint maps.  
Each objective $\C{J}_i$ is a real-valued function defined on an open set $O = O_1 \times \cdots \times O_N$, where $O_i$ is an open subset of $U_i$ containing $U_i^\ad$.  
For simplicity, we assume that the effective domain of each $\C{J}_i$ is contained in $U^\ad$.

Each coupling constraint map $\B{U}_i$ is a set-valued map from $O_{-i}$ to $O_i$.  
The set $\B{U}_i(u_{-i})$ represents the strategies available to player $i$ when the other players choose $u_{-i}$.  
We denote by $\B{U}$ the product map
$
  \B{U} \coloneqq \B{U}_1 \times \cdots \times \B{U}_N.
$A game with the structure above will be denoted by $G = (\B{U},\C{J})$.

We are interested in games where the coupling constraints are induced by the state equation.  
For each $i = 1,\dots,N$, we assume that $\B{U}_i$ has the form
\begin{equation}
  \label{set_map}
  \B{U}_i(u_{-i})
  = \{u_i \in O_i \,;\, \C{S}_i(u) \in \C{K}_i\},
\end{equation}
where $y_i \coloneqq \C{S}_i(u)$ denotes the solution of \eqref{abs1_dis_in}, whose existence is guaranteed by semigroup theory and \Cref{ass_a} (see \Cref{state_prop}), and $\C{K}_i$ is a nonempty, convex set which is assumed to be closed in some topology specified later.

To ensure that each player can solve their optimization problem for any fixed $u_{-i} \in U_{-i}$, we impose the following assumptions.

\begin{ass}\label{ass_c}
\ 
\begin{itemize}
  \item[\asspartlabel{ass_c-a}]
    For each $i = 1,\dots,N$, the functional $\C{J}_i : O \to \mathbb{R}$ is weakly lower semicontinuous, and for each $u_{-i} \in O_{-i}$ the map
    $\C{J}_i(\cdot,u_{-i}) : O_i \to \mathbb{R}$ is continuous.
  \item[\asspartlabel{ass_c-b}]
    For each $i = 1,\dots,N$, the set $U_i^\ad$ is nonempty, bounded, closed and convex.
  \item[\asspartlabel{ass_c-d}]
    There exists $u \in U^\ad$ such that $u \in \B{U}(u)$.
\end{itemize}
\end{ass}

\subsection{Constraint topology and regularity and general existence}

To guide the discussion, we now fix $W$ as a Banach space (not necessarily reflexive) such that $\C{K}_i \subset W$. As will be shown in this section, weak closedness of the maps $\B{U}_i(\cdot): U_\mt^\ad \rightrightarrows U_i$ indicates the minimal topological requirements on $W$, because weak closedness is, in the particular form of \eqref{set_map}, equivalent to weak upper semicontinuity. The latter is known to be a fundamental requirement for existence, see \cite{debreu_chapter_1982,pang_quasi-variational_2005,zame_competitive_1987} and, in particular, our previous work \cite{bongarti2025structureversusregularitysetvalued}. We make this precise in the lemma below.

\begin{proposition}\label[proposition]{lemma_closednessequiv}
    Assume that $\C{S}_i:U\to W$ is well-defined and that $\C{K}_i$ is closed in $W.$ Then, for each $i = 1, \cdots, N$, if the map $\C{S}_i : U \to W$ is weak-to-strong continuous, then the map $\B{U}_i(\cdot): U \rightrightarrows U_i$ is weakly closed.
\end{proposition}
\begin{proof}
    Let $(u^n)$ be a sequence in ${\rm Gr}(\B{U}_i)$, i.e., $u_i^n \in \B{U}_i(u_\mt^n)$ or $\C{S}_i(u^n) \in \C{K}_i$ for each $n$, and assume that $u^n \rightharpoonup u$ in $U$. By weak-to-strong continuity of $\C{S}_i$ we have $\C{S}_i(u^n) \to \C{S}_i(u)$ and, by closedness of $\C{K}_i$, we have $u_i \in \B{U}_i(u_\mt)$, and this ends the proof. 
\end{proof}

We notice that the \say{strong} part of the continuity of $\C{S}_i$ is essential in the previous lemma, because $W$ is not necessarily reflexive, hence convexity of $\C{K}_i$ does not guarantee weak closedness from closedness. If reflexivity of $W$ is available, then weak-to-weak continuity of $\C{S}_i$ would be enough.

Next, we want to show that with sufficient compactness of the spaces, one can show that $\C{S}_i$ is weak-to-strong continuous. 

\begin{proposition}\label[proposition]{lemma_weaktostrong_gen}
    Let $X$ be a Banach space such that $X \xhookrightarrow{c} W$, $\C{S}_i: U \to X$ is well-defined through \emph{\eqref{vop_formula}} and Lipschitz. Then, $\C{S}_i: U \to W$ is weak-to-strong continuous. 
\end{proposition}
\begin{proof}
    Since $\C{S}_i: U \to X$ is an affine bounded (since Lipschitz) map, it is automatically weak-to-weak continuous, hence $\C{S}_i: U \to W$ weak-to-strong continuous for any $W$ such that $X \xhookrightarrow{c} W.$
\end{proof}

\Cref{lemma_closednessequiv,lemma_weaktostrong_gen} lead to the minimal condition necessary to study such state constraints.

\begin{ass}\label{min_constraint}
   There exists Banach spaces $X$ and $W$ such that \begin{itemize}
       \item[\bf (i)] $X \xhookrightarrow{c} W$;
       \item[\bf (ii)] $\C{S}_i: U \to X$ is well-defined through \emph{\eqref{vop_formula}} and Lipschitz; and
       \item[\bf (iii)] $\C{K}_i \subset W$ is nonempty, closed and convex.
   \end{itemize}
\end{ass}

\begin{theorem}\label{upper_semi}
    Under \Cref{min_constraint}, the map $U_i(\cdot): U_\mt^\ad \rightrightarrows U_i^\ad$ is weakly upper semicontinuous.
\end{theorem}
\begin{proof}
    Let $F: U_\mt^\ad \rightrightarrows U_i$ be defined as $F \equiv U_i^\ad$, hence $F$ is weakly upper semicontinuous (as a constant function) and, from \Cref{ass_c-b}, $F(u_\mt) = U_i^\ad$ is weakly compact for all $u_\mt \in U_\mt^\ad.$ Furthermore, from \Cref{lemma_weaktostrong_gen} combined with \Cref{lemma_closednessequiv} it follows that the map $\B{U}_i: U_\mt^\ad \rightrightarrows U_i$ is weakly closed. The result then follows from \cite[Proposition 1.4.9, p. 42]{aubin_set-valued_2009}.
\end{proof}

\begin{proposition}\label[proposition]{graph_convex}
    The bundle map $\B{U}: U^\ad \to U^\ad$ is graph-convex.
\end{proposition}
\begin{proof}
    Fix $\lambda \in (0,1).$ Let $u,v \in {\rm Gr}(U^\ad \cap \B{U})$, which means $u,v \in U^\ad$ and $\C{S}_i(u), \C{S}_i(v) \in \C{K}_i$. Since $\C{K}_i$ is convex, it follows that $$(1-\lambda)\C{S}_i(u) + \lambda \C{S}_i(v) \in \C{K}_i.$$ By using the integral representation \eqref{vop_formula} we see that $$(1-\lambda)\C{S}_i(u) + \lambda \C{S}_i(v) = \C{S}_i((1-\lambda)u + \lambda v) \in \C{K}_i$$ whereby \Cref{ass_c-b} implies that $(1-\lambda)u + \lambda v \in {\rm Gr}(U^\ad \cap \B{U}).$
\end{proof}

\begin{theorem}\label{exist_gen_N}
  The game $G = (\B{U},\C{J})$ has a generalized Nash equilibrium.
\end{theorem}
\begin{proof}
The map $\mathcal{U}^\text{ad} \cap \mathcal{U}(\cdot)$ is graph-convex by proposition \ref{graph_convex}. The result then follows because proposition \ref{lemma_closednessequiv} and theorem \ref{upper_semi} show that Theorem~5.1 in \cite{bongarti2025structureversusregularitysetvalued} applies.
\end{proof}

\subsection{The Choice of the space $W$}\label{space_choice}

Based on \Cref{state_prop} and depending on how regular our data is, we distinguish two situations:

\begin{itemize}
  \item[\bf (I)] a {\bf classical state setting}, where
  \[
    y = \C{S}_i(u) \in C^1([0,T];Y)\cap C([0,T];\C{D}(\C{A}));
  \]
  \item[\bf (II)] a {\bf dual classical state setting}, where
  \[
    y = \C{S}_i(u) \in C^1([0,T];\C{D}(\C{A}^*)')\cap C([0,T];Y).
  \]
\end{itemize}

In both cases, the state constraint set for player $i$ will be a nonempty,
convex, and closed subset $\C{K}_i \subset W$ for a suitable choice of $W$.

In the classical state setting, we choose $W = C([0,T];E)$ for any $E$ such that
\begin{equation}\label{cl_inc}
  \C{D}(\C{A}) \xhookrightarrow{c} E \hookrightarrow Y
\end{equation}
so we can conclude, from the Aubin--Lions lemma \cite[Corollary 4, p. 84]{simon_compact_1986} that
\[
  C^1([0,T];Y)\cap C([0,T];\C{D}(\C{A})) \xhookrightarrow{c} W.
\]

Similarly, for the dual state setting, we need to assume that $Y \xhookrightarrow{c} \C{D}(\C{A}^*)'$. We choose $W = C([0,T];E)$ for any $E$ such that
\begin{equation}\label{cl_dinc}
  Y \xhookrightarrow{c} E \hookrightarrow \C{D}(\C{A}^*)',
\end{equation}
so we can conclude, again from the Aubin--Lions lemma, that
\[
  C^1([0,T];\C{D}(\C{A}^*)')\cap C([0,T];Y) \xhookrightarrow{c} W.
\]

In both cases, the conclusions of \Cref{lemma_weaktostrong_gen} are at our disposal.

Before we move forward, we include an example illustrating the ideas above.

\begin{example}\label{ex:2p_transport}
Let $T>0$, $c>0$, and $\Omega=(0,1)$.
Consider the first--order hyperbolic PDE
\begin{equation}\label{eq:transport}
  y_t + c\,y_x = p(t)y + u_1 + u_2 + f,
  \qquad \text{in} \ (0,T)\times(0,1),
\end{equation}
with initial condition $y(0,x)=y_0(x)$ and an inflow boundary condition $y(t,0)=g(t).$

This is a \emph{wave/transport} model in first order form and can be written in the form \eqref{abs1_dis_in}. We set
\[
  Y \coloneqq L^2(0,1),\qquad
  U_1=U_2\coloneqq L^2\bigl(0,T;L^2(0,1)\bigr),\qquad
  U=U_1\times U_2,
\]
and define $\C{A}:\C{D}(\C{A})\subset Y\to Y$ by
\[
  \C{A}y = -c\,y_x,
\]
with domain
\[
  \C{D}(\C{A})=\{y\in H^1(0,1)\,;\ y(0)=0\}.
\]
Then $\C{A}$ generates a $C_0$--semigroup on $Y=L^2(0,1)$ (the left/right shift semigroup). In that case, for example, we can take $E = H^s(0,T)$ for all $s \in (0,1)$ or $E = C([0,1]).$ 

If we have the remaining of the data complying with \Cref{classical_setting}, then \Cref{min_constraint} is satisfied and any game under the conditions discussed above will have an equilibrium. In particular, pointwise state constraints of the form \begin{equation}
    a \leqslant y(t,x) \leqslant b, \ \text{for all} \ (t,x) \in [0,T] \times [0,1]
\end{equation} ($a,b \in \mathbb{R}$, $a < b$) can be handled. \EX \end{example}

\subsection{Conditions on non-standard form and partial compactness}

In this subsection we want to treat two cases in which the chain of embeddings \eqref{cl_inc} and \eqref{cl_dinc} do not hold or hold only partially. If \eqref{cl_inc} does not hold for $Y$ but it does for any space $Z$ such that $Y \hookrightarrow Z$, we can essentially use the same argument by noticing that $C^1([0,T],Y) \hookrightarrow C^1([0,T];Z).$ A similar comment holds true when \eqref{cl_dinc} does not hold for $\C{D}(\C{A}^*)'.$ This rather simple observation is used in the example below.

\begin{example}\label{ex:partial_compactness_bounded_A}
Let $T>0$ and $\Omega=(0,1)$. We consider a two--player game where the state satisfies a \emph{spatially distributed but non-regularizing} linear evolution:
\begin{equation}\label{eq:bounded_A_state}
  y_t(t) = \C{A}y(t) + p(t)\,y(t) + u_1(t) + u_2(t) + f(t)
  \qquad\text{in }(0,T),
  \qquad y(0)=y_0.
\end{equation}
Here $p\in C^1([0,T])$ is scalar, and the controls are distributed in space:
\[
  U_1=U_2 \coloneqq L^2\bigl(0,T;L^2(\Omega)\bigr),
  \qquad U=U_1\times U_2.
\]
We set the state space
\[
  Y \coloneqq L^2(\Omega),
\]
and choose a \emph{bounded} generator $\C{A}\in\C{L}(Y)$, for instance
\[
  [\C{A}y](x) \coloneqq a(x)\,y(x)
  \quad\text{with } a\in L^\infty(\Omega).
\]
Then $\C{A}$ generates a $C_0$--semigroup on $Y$ given by $S(t)y=e^{t a(\cdot)}y$.

Because $\C{A}$ is bounded we have $\C{D}(\C{A})=Y$ with equivalent graph norm, hence $\eqref{cl_inc}$ would not hold. Let
\[
  Z \coloneqq H^{-1}(\Omega),
  \qquad E \coloneqq H^{-1/2}(\Omega),
  \qquad W \coloneqq C\bigl([0,T];E\bigr).
\]
Then
\[
  Y=L^2(\Omega)\hookrightarrow Z=H^{-1}(\Omega),
  \qquad
  L^2(\Omega)\xhookrightarrow{c}H^{-1/2}(\Omega),
  \qquad
  H^{-1/2}(\Omega)\hookrightarrow H^{-1}(\Omega).
\]
If the data satisfy the classical assumptions (e.g.\ $y_0\in Y$, $f\in H^1(0,T;Y)$, and the control enters through $\C{B}(u)=u_1+u_2$),
then the setting of \Cref{min_constraint} holds with
\[
  X \coloneqq C^1([0,T];Z)\cap C([0,T];Y),
  \qquad W=C([0,T];E),
\]
even though \eqref{cl_inc} fails for $Y$.

This setting allows us to impose closed convex constraints in $W$, such as the pointwise in time bound in the weaker norm
\[
  \C{K}_i
  \coloneqq
  \Bigl\{y\in C([0,T];H^{-1/2}(\Omega)):\ \|y(t)\|_{H^{-1/2}(\Omega)}\le R_i \ \forall t\in[0,T]\Bigr\},
\]
with $R_i>0$.
\EX
\end{example}

In some hyperbolic models the full generator $\C{A}$ on the state space $Y$ does
not have compact resolvent, so that $\C{D}(\C{A})\xhookrightarrow{c}Y$ fails.  
However, it is often the case that only one component of the state is
constrained, and this component enjoys better spatial regularity with a
compact embedding.  
In this subsection we show that the equilibrium theory can still be applied
in such situations by projecting the state onto the constrained component.

Let the state space be a product
\[
  Y = Y_{\mathrm{a}} \times Y_{\mathrm{b}},
\]
and assume that $\C{A}:\C{D}(\C{A})\subset Y\to Y$ generates a $C_0$--semigroup on $Y$
but does not necessarily have compact resolvent.  
We denote the canonical projection onto the component, say $Y_{\mathrm{a}}$, where the constrained variable is taken from, by
\[
  \Pi: Y \to Y_{\mathrm{a}},\qquad
  \Pi(y_{\mathrm{a}},y_{\mathrm{b}}) = y_{\mathrm{a}}.
\]

Let $\C{S}_i : U \to C([0,T];Y)$ be the strategy--to--state map defined by the hyperbolic state equation. We assume that for each $u\in U$ the constrained component $z \coloneqq \Pi(\C{S}_i(u))$ belongs to a space $C([0,T];E_{\mathrm{a}})$, where $Y_{\mathrm{a}} \xhookrightarrow[]{c} E_{\mathrm{a}}$ is a compact embedding.  

We then introduce the reduced state map
\[
  \hat{\C{S}}_i \coloneqq \Pi \circ \C{S}_i : U \to C([0,T];E_{\mathrm{a}}).
\]

If $\hat{\C{S}}_i$ is Lipschitz from $U$ into a space $X_{\mathrm{a}}$
such that $X_{\mathrm{a}} \xhookrightarrow[]{c} C([0,T];E_{\mathrm{a}}),$
then the weak--to--strong continuity arguments from \Cref{lip_cont} and the
abstract existence theory apply directly to $\hat{\C{S}}_i$, even though
$\C{D}(\C{A})\not\xhookrightarrow{c}Y$.  
In particular, state constraints depending only on the component
$z \in Y_{\mathrm{a}}$ can be treated without requiring compactness of the full generator.

\begin{example}Consider the third--order in time equation
\begin{equation}\label{eq:mgt}
  z_{ttt} + \alpha z_{tt} - c^2 \Delta z - b \Delta z_t = u_1 + u_2
  \quad\text{in } (0,T)\times\Omega,
\end{equation}
with suitable initial conditions and homogeneous Dirichlet boundary conditions on a bounded smooth domain
$\Omega\subset\BR^d$, which is a well known linear model for the propagation of acoustic waves; see e.g. \cite{bongarti_singular_2020}. 
A standard first--order formulation introduces
$
  y(t) = (z(t),z_t(t),z_{tt}(t)),
$
and yields an abstract Cauchy problem
\[
  y_t(t) = \C{A} y(t) + \C{B}(u)(t), \qquad y(0) = y_0,
\]
with
\[
\C{A}[z,z_t,z_{tt}] = \left[z_t, z_{tt}, -\alpha z_{tt} + c^2\Delta z + b\Delta z_t\right]
\]
\[
\C{D}(\C{A}) = \{(y_1,y_2,y_3) \in Y: c^2 y_1 + by_2 \in H^2(\Omega) \cap H_0^1(\Omega)\}
\]
on a product space of the form
$
  Y = Y_{\mathrm{a}} \times Y_{\mathrm{b}},
$
where, for instance,
\[
  Y_{\mathrm{a}} =  H_0^1(\Omega), \qquad 
  Y_{\mathrm{b}} = H_0^1(\Omega) \times L^2(\Omega).
\]
It is well known that the corresponding generator $\C{A}$ need not have compact
resolvent \cite{marchand_abstract_2012}, so that $\C{D}(\C{A})\xhookrightarrow{c}Y$ fails.

However, defining $\Pi(z,z_t,z_{tt}) = z$ and setting $E_{\mathrm{a}} = C(\overline\Omega).$
Then, for each $u\in U$, we have
\[
  z = \Pi(\C{S}_i(u)) \in C([0,T];C(\overline\Omega)).
\]

The state constraints now act only on $z$, for example
\[
  a_i(x) \le z(t,x) \le b_i(x)
  \quad\text{for all } (t,x)\in[0,T]\times\overline\Omega,
\]
with given continuous bounds $a_i,b_i:\overline\Omega\to\BR$.  
In terms of the reduced map $\hat{\C{S}}_i = \Pi\circ \C{S}_i$, we define
\[
  \C{K}_i^{\mathrm{a}}
  \coloneqq
  \Bigl\{
    z\in C([0,T];C(\overline\Omega))
    \,;\,
    a_i(x)\le z(t,x)\le b_i(x)
    \text{ for all }(t,x)\in[0,T]\times\overline\Omega
  \Bigr\},
\]
and the admissible state set
\[
  \B{U}_i(u_\mt)
  \coloneqq
  \bigl\{u \in U: \ \Pi(\C{S}_i(u)) \in \C{K}_i^{\mathrm{a}}\bigr\}.
\]

Crucially, the equilibrium existence theory only requires compactness for the
map
\[
  \hat{\C{S}}_i : U \to C([0,T];E_{\mathrm{a}}),
\]
not for the full map $\C{S}_i:U\to C([0,T];Y)$.  
Thus the lack of compact resolvent for $\C{A}$ does not prevent us from handling
hyperbolic models of the form \eqref{eq:mgt} as long as the constraints act
only on the component $z$ for which we can identify a compact embedding
$Y_{\mathrm{a}}\xhookrightarrow{c}E_{\mathrm{a}}$.
\end{example}

\section{Two examples}\label{sec:two_examples} 
In this section we provide two examples on which our results can be used.
\begin{example}[\bf The wave equation with interior control]
    Let $\Omega \subset \BR^d$ $(d = 2$ or $3)$ be a smooth domain and consider the wave equation with interior control:
    \begin{equation}
        \label{ex_wave1}
        \begin{cases}
            z_{tt} - \Delta z = u, & \text{in} \ (0,T) \times \Omega, \\
            z = 0, & \text{on} \ (0,T) \times \partial \Omega, \\
            z(0,\cdot) = z_0, \ z_t(0,\cdot) = z_1, & \text{in} \ \Omega.
        \end{cases}
    \end{equation}
    Such an equation can be put in the form \eqref{abs1_dis_in} as follows: consider the Laplacian operator $A=-\Delta$ with domain $\D{A} = H^2(\Omega) \cap H_0^1(\Omega).$ We set $ y = (z,z_t)^\top$, take $Y = H_0^1(\Omega) \times L^2(\Omega)$, $U = L^2(\Omega)$ and define $$\C{A} = 
    \begin{bmatrix}
        0 & I \\
        -A & 0
    \end{bmatrix}, \qquad \D{\C{A}} = \D{A} \times H_0^1(\Omega),$$ 
    $$\C{B} = 
    \begin{bmatrix}
        0 \\
        I
    \end{bmatrix}, \qquad \C{P}(t) \equiv 0.$$
    It is easy to see that \eqref{abs1_dis_in} with the above operators is an abstract representation of \eqref{ex_wave1}. Similarly (although maybe under different conditions), the wave equation with boundary Dirichlet/Neumann control can also be put in the form \eqref{abs1_dis_in}. For further details see \cite{lasiecka_control_2000}. \EX
\end{example}

\begin{example}[\bf Personalized advertising] \label{advertising} Let $\Omega = (0,1)$ and consider the advection equation 
\begin{equation}
    \label{advection}
    \begin{cases}
        a_t + a_x = -\delta a + \lambda u & \text{in} \ (0,T), \times \Omega, \\
        a(0,x) = a_0(x) > 0 & \text{in} \ \Omega, \\
        a(t,0) = 0.
    \end{cases}
\end{equation}
According to \cite{machowska_closed-loop_2022}, the above equation can be used as a model for the evolution of a company's goodwill. It is assumed that the spatial variable models a continuum of segments that the market has. The quantity $x \in [0,1]$ represents the usage experience of consumers in this segment. If we therefore consider the solution $a: (0,T) \times \overline{\Omega} \to \mathbb{R}$ of \eqref{advection}, then $a(t,x)$ represents the goodwill of this company at time $t$ and (for short) segment $x.$

The control variable $u: (0,T) \times \overline\Omega \to [0,K]$ is such that $u(t,x)$ represents the advertising policy aimed at segment $x$ at time $t$ and the coefficient $\lambda = \lambda(x)$ represents the effectiveness of such a strategy. The constant $K$ represents the upper limit of the company's budget for advertising for segment $x$. Meanwhile, since goodwill naturally depreciates, the coefficient $\delta = \delta(a)$ represents this amortization.

The zero boundary condition in this context means that the market is closed (for new customers), a condition that we only assume here for the sake of simplicity and illustration. See \cite{machowska_closed-loop_2022} for a further discussion of the modeling of the boundary condition. 

A simple application of Lumer--Phillips Theorem \cite[Theorem 4.3, p. 14]{pazy_semigroups_1983} implies that the operator $\C{A} \coloneqq -\partial_x$ with domain $\C{D}(\C{A}) \coloneqq \{a \in H^1(\Omega); a(0) = 0\}$ generates a $C_0$--semigroup of contractions in $L^2(\Omega).$ The details remaining to put \eqref{advection} into the form of \eqref{abs1_dis_in} are straightforward. \EX
\end{example}

\section{Characterization of equilibria}\label{sec:characterization} 

In this section we provide a characterization of equilibrium strategies via first-order optimality conditions. We achieve this in a rather abstract setting that covers our general cases. 

The problem that player $i$ needs to solve reads:
\begin{equation}\label{gameic}\tag{$P_i(u_\mt)$}
    \begin{cases}
        \text{Given } u_\mt \in U_\mt^\ad; \\
        \text{Minimize } \C{J}_i(u_i, u_\mt) \ \text{subject to } u_i \in U_i^\ad \cap \B{U}_i(u_\mt),
    \end{cases}
\end{equation}
where \begin{equation}
    \label{constraint_set_chac}
    \B{U}_i(u_\mt) = \{u_i \in U_i^\ad; \ \C{S}_i(u) \in \C{K}_i\}
\end{equation} with \Cref{min_constraint} in place. 

By definition, $u^*$ is a generalized Nash equilibrium if and only if, for all $i = 1, \cdots, N$, $u_i^*$ solves the problem 
\begin{equation}\label{gameic2}
        \text{Minimize } \C{J}_i(u_i, u_\mt^*) + I_{U_i^\ad}(u_i) + I_{\mathcal{K}_i}(\mathcal{S}_i(u_i, u_\mt^*)) \ \text{over} \ u_i \in U_i,
\end{equation} where $I_C$ denotes the indicator function associated with a non-empty, convex and closed set $C$. Due to the convexity of $\mathcal{J}_i$, $U_i^\ad$ and $\mathcal{K}_i$, this minimization problem is equivalent to finding $u^*$ such that
\begin{equation}
    \label{0diff}
    0 \in \partial\left(\C{J}_i(\cdot, u_\mt^*) + I_{U_i^\ad}(\cdot) + I_{\mathcal{K}_i}(\mathcal{S}_i(\cdot, u_\mt^*))\right)(u_i^*).
\end{equation}

Without further information, \eqref{0diff} is the most we can say about the optimality of $u_i^*$. Recall that $\C{J}_i(\cdot, u_\mt^*)$, $I_{U_i^\ad}$ and $I_{\C{K}_i} \circ \C{S}_i(\cdot, u_\mt^*)$ are proper convex functions due to \Cref{ass_c-d}. Hence, according to \cite[Theorem 1, p. 200]{ioffe_theory_2010}) we could use the sum rule of subdifferential calculus in \eqref{0diff} provided either $I_{U_i^\ad}$ or $I_{\C{K}_i} \circ \C{S}_i(\cdot, u_\mt^*)$ is continuous at a point of $U_i^\ad \cap \B{U}(u_\mt^*).$ Taking this for granted for the time being we rewrite \eqref{0diff} as
\begin{align}\label{1diff}
    0 \in \partial_i(\C{J}_i)(u^*) + \C{N}_{U_i^\ad}(u_i^*) + \partial_i (I_{\C{K}} \circ \C{S}_i(\cdot, u_\mt^*))(u_i^*)
\end{align} where we use $\partial_i$ to emphasize that the subdifferential is taken with respect to the decision variable of player $i$, and $\C{N}_C$ denotes the normal cone to $C \subset U_i$ defined at some $u_i\in C$ by
$$
\C{N}_C(u_i)=\{\mu_i\in U_i^*:\mu_i(v_i-u_i)\leqslant 0\:\forall v_i\in C\}.
$$

We would like now to apply the chain rule to the last term of \eqref{1diff}. For that we need to assume that $I_{\C{K}_i} \circ \C{S}_i(\cdot, u_\mt^*)$ has a continuity point (see \cite[Theorem 2, p. 201]{ioffe_theory_2010}) or that $\C{S}_i^0: U \to W$ is surjective (see \cite[Theorem 10.19, p. 202]{ioffe_theory_2010}). The second condition is often not true in our setting.

\begin{definition}\label{cont_point}
    We say that a point $u \in U$ is a continuity point w.r.t player $i$ if $u_i$ is a continuity point of the map $I_{\C{K}_i} \circ \C{S}_i(\cdot, u_\mt)$, i.e., if there exists $\varepsilon > 0$ such that
    \begin{equation}
        B_\varepsilon^i(0) \subset \C{S}_i(u) - \C{K}_i
    \end{equation}
    where $B_\varepsilon^i(0)$ denotes the ball in $W$ of radius $\varepsilon$ and centered at zero.
\end{definition}

Now, if for each $i=1, \cdots, N$, $u^*$ is a continuity point w.r.t. player $i$, we use the chain rule to obtain \begin{align*}
    0 \in 
    \partial_i(\C{J}_i)(u^*) + \C{N}_{U_i^\ad}(u_i^*) + [\C{S}_i^0]^*\C{N}_{\C{K}_i-\C{S}_i(0)}(\C{S}_i^0(u^*)). 
\end{align*} 

We then have the following optimality conditions:
\begin{theorem}\label{abs_opt}
    Let \Cref{min_constraint} hold, let $u^* \in \B{U}(u^*)$ be a Nash equilibrium and assume it to be a continuity point w.r.t. player $i$ for each $i.$ Then there exists $\Phi = (y_1^*, \cdots, y_N^*) \in X^N, \lambda \in U^*$ and $\mu = (\mu_1, \cdots, \mu_N) \in [W^*]^N$ such that the system
    \begin{equation}
        \label{opt_sys_abs} \begin{cases}
            y_i^* = \C{S}_i(u^*); \\
            \lambda_i \in \C{N}_{U_i^\ad}(u^*); \\
            \mu_i \in \C{N}_{\C{K}_i}(\C{S}_i(u^*));\\
            0 \in \partial_i(\C{J}_i)(u^*) + \lambda_i + [\C{S}_i^0]^*\mu_i \\
        \end{cases}
    \end{equation}
    is satisfied for each $i = 1, \cdots, N.$ Conversely if $(\Phi,\lambda,\mu) \in X^N \times U^* \times [W^*]^N$ is such that $(y_i^*,\lambda_i, \mu_i)$ solves \eqref{opt_sys_abs} for each $i$, then $u^*$ is a Nash equilibrium.
\end{theorem}
\begin{proof}
    The proof follows basically from our discussion before the theorem. In fact, assuming the existence of a point of continuity, $u^* \in \B{U}(u^*)$ is a Nash equilibrium if and only if \begin{align*}
    0 \in 
    \partial(\C{J}_i)(u^*) + \C{N}_{U_i^\ad}(u_i^*) + [\C{S}_i^0]^*\C{N}_{\C{K}_i-\C{S}_i(0)}(\C{S}_i^0(u^*)). 
\end{align*}
\end{proof}
However, it is important to notice that the system of optimality conditions we derive here, in its abstract form, provides very little information towards an explicit characterization of the equilibrium. See the example in the next section.

If more is known about the objective functions, then the optimality system \eqref{opt_sys_abs} can be improved.

\begin{example}
    It is very common in optimization to consider objective functions of the form \begin{equation}
        \label{obj_sep} \C{J}_i(u) = \C{J}_i^1(\C{S}_i(u)) + \C{J}_i^2(u_i).
    \end{equation}
    In this case we have \begin{align*}
        \partial_i(\C{J}_i(\cdot, u_\mt^*))(u_i^*) &= \partial_i(\C{J}_i^1 \circ \C{S}_i(\cdot, u_\mt^*))(u_i^*) + \partial_i (\C{J}_i)(u_i^*) \\ &= [\C{S}_i^0]^* \partial_i\C{J}_i^1(\C{S}_i(u^*)) + \partial_i \C{J}_i^2(u_i^*).
    \end{align*}
    Then, by introducing the adjoint state 
    \begin{equation}
        p_i^* \in [\C{S}_i^0]^*(\partial_i \C{J}_i^1(\C{S}_i(u^*))+\mu_i), \qquad \mu_i \in \C{N}_{\C{K}_i}(\C{S}_i(u^*)),
    \end{equation} the optimality system \eqref{opt_sys_abs} then reads 
    \begin{equation}
        \label{opt_sys_abs2} \begin{cases}
            y_i^* = \C{S}_i(u^*); \\
            p_i^* \in [\C{S}_i^0]^*(\partial_i \C{J}_i^1(\C{S}_i(u^*))+\mu_i); \\
            \lambda_i \in \C{N}_{U_i^\ad}(u^*); \\
            \mu_i \in \C{N}_{\C{K}_i}(\C{S}_i(u^*));\\
            0 \in \partial_i(\C{J}_i^2)(u_i^*) + \lambda_i + p_i^*.
        \end{cases}
    \end{equation} \EX
\end{example}

\section{A finite game constrained by the linearized isothermal Euler system for gas dynamics on a network} \label{ex_gas}

We now put the abstract machinery developed above to work on a concrete example: a noncooperative game in which a gas market is coupled with a physical transport network. Gas moves through the pipes according to a linearized isothermal Euler system of the form \eqref{iso2ex}, and each player perceives this shared physical system a little differently, since every player relies on its own model of how the flow behaves on each pipe. Because of this asymmetry the game cannot be recast as a single optimization problem driven by one common objective (a so-called potential game); it genuinely requires the graph-convexity approach built in the previous sections.

\subsection{The state equation and network topology}

 \begin{figure}[ht]
	   \centering
	   \begin{tikzpicture}[scale=.55, transform shape]
		      \node [circle, draw,color=blue] (g) at (1, 3) {$v_1$} ;
		      \node [circle, draw] (c) at (4, 3) {$v_2$} ;
		      \node [circle, draw, color=red] (d) at (7, 1.5) {$v_4$} ;
		      \node [circle, draw,color=red] (e) at (7, 4.5) {$v_3$} ;
		      \draw [->,line width=1] (g) -- (c) node[pos=.3,above] {$e_1$} ;
		      \draw [->,line width=1] (c) -- (d) node[pos=.3,above] {$e_3$};
		    \draw [->,line width=1] (c) -- (e) node[pos=.3,above] {$e_2$};
	   \end{tikzpicture}
      \caption{$3$-pipes gas network}
	   \label{net1}
    \end{figure}
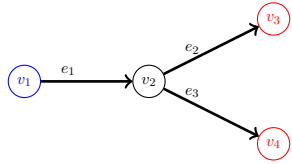 

The gas network is modeled as a directed graph $\mathcal{G} = (\mathcal{V}, \mathcal{E})$: the vertices $\mathcal{V} = \{v_1, v_2, v_3, v_4\}$ are the junctions where pipes meet or where gas enters or leaves the system, and the edges $\mathcal{E} = \{e_1, e_2, e_3\}$ are the pipes themselves; see \cref{net1}. To keep the notation light, every pipe is normalized to unit length, so a point along pipe $e_k$ is described by a single coordinate $x \in (0,1)$, with $x=0$ and $x=1$ marking its two ends.

The four junctions play different roles. Vertex $v_2$ is an internal node, where three pipes meet and gas is simply redistributed; $v_1$ is where gas is injected into the network (producer); and $v_3$, $v_4$ are where gas is withdrawn (consumer or wholesaler). We record this partition as
\begin{equation}
    \label{vertices}
    \mathcal{V}^\circ = \{v_2\}, \qquad \mathcal{V}_-^\partial = \{v_1\}, \qquad \mathcal{V}_+^\partial = \{v_3,v_4\}.
\end{equation}
Physically, the network therefore has three points of control: an operator can set the gas pressure at the inlet $v_1$, while two downstream suppliers can each set how much gas they draw off at $v_3$ and $v_4$. These three quantities will become the three players' strategies below.

On each pipe $k \in \{1,2,3\}$, over a time horizon $[0,T]$, the state of the gas is described by two quantities: the pressure $p^k(t,x)$ and the mass flux $q^k(t,x)$, which measures how much gas moves past the point $x$ at time $t$. Under isothermal conditions, and after linearizing around a reference equilibrium, these two quantities obey the system
\begin{equation}\label{iso2ex}
    \begin{cases}
        p_t^k + c^2q_x^k = 0, & t \in (0,T), \; x \in (0,1), \\[1mm]
        q_t^k + p_x^k = f_1^k p^k + f_2^k q^k, & t \in (0,T), \; x \in (0,1),
    \end{cases}
\end{equation}
where $c>0$ is the speed of sound. The first equation is a conservation law: pressure changes in time are balanced by how the flux varies along the pipe. The second is a momentum balance, and the terms on the right, weighted by the coefficients $f_1^k, f_2^k$, capture the linearized effects of pipe friction and elevation change on pipe $k$.

So far the three pipes evolve according to independent copies of the same equation; what turns them into a single network is what happens at the junction $v_2$, where pipes $1$, $2$ and $3$ meet. Physically, gas cannot pile up or vanish at a junction, and its pressure cannot jump discontinuously across it. These two facts, continuity of pressure and conservation of mass (Kirchhoff's law), are encoded as
\begin{align}
    \label{pressure_cont}
    &p^1(t,1) = p^2(t,0) = p^3(t,0), \quad &&t \in (0,T), \\
    \label{flow_cont}
    &q^1(t,1) - q^2(t,0) - q^3(t,0) = 0, \quad &&t \in (0,T),
\end{align}
where we have implicitly assumed that all three pipes share the same cross-sectional diameter. Finally, the system starts from a given initial pressure and flux profile on each pipe:
\begin{equation} \label{init_cond}
    p^k(0,x) = p_0^k(x) \quad \text{and} \quad q^k(0,x) = q_0^k(x), \quad k = 1, 2, 3.
\end{equation}

\subsection{Boundary homogenization and operator formulation}

The abstract framework of \eqref{abs1_dis_in} is written for systems driven by a distributed source term, with zero data on the boundary. In our gas network, however, the three players act directly on the boundary: they each fix a pressure or a flux at one of the network's open ends. To bring the model into the required form we use a standard device from boundary control theory known as \emph{lifting}: we subtract from the true state a simple, explicitly known function that carries all of the nonzero boundary values, so that what remains solves a problem with homogeneous boundary data, at the price of an extra forcing term generated by the lifting itself.

Concretely, let $\mathbf{u}(t) = [u_1(t), u_2(t), u_3(t)]^\top \in \mathbb{R}^3$ collect the three boundary controls: $u_1(t) = p^1(t,0)$ is the pressure the operator sets at the network's inlet, while $u_2(t) = q^2(t,1)$ and $u_3(t) = q^3(t,1)$ are the fluxes withdrawn by the two suppliers at their respective outlets.

We collect all six boundary values of the network state, pressure and flux at each of the three pipe-ends touching the boundary, into a single vector $\mathbf{b}(t) \in \mathbb{R}^6$; only three of its six entries are actually controlled, the rest being left as zero placeholders since no condition is imposed there. The matrix $M$ simply records which slot each control feeds into:
\begin{equation}\label{bdata}
    \mathbf{b}(t) = [p^1(t,0), 0, 0, q^2(t,1), 0, q^3(t,1)]^\top, \qquad 
    M = \begin{bmatrix}
        1 & 0 & 0 \\ 0 & 0 & 0 \\ 0 & 0 & 0 \\ 0 & 1 & 0 \\ 0 & 0 & 0 \\ 0 & 0 & 1
    \end{bmatrix}.
\end{equation}
To turn this boundary vector into something that can be subtracted from the actual state, we need to spread it out over each pipe. For $\xi \in (0,1)$ define
\begin{equation}
    \mathbb{B}_0^k = \begin{bmatrix} 0 & -c^2 \\ 1 & 0 \end{bmatrix}, \qquad 
    \mathbb{B}_1^k(\xi) = \begin{bmatrix} 1-\xi & 0 \\ 0 & \xi \end{bmatrix}, \quad k=1,2,3.
\end{equation}
Here $\mathbb{B}_1^k(\xi)$ is nothing more than a linear interpolation in space, decaying from full weight at one end of the pipe to zero at the other. Stacking these over the three pipes gives the network-wide matrices $\mathbb{B}_0 = \text{diag}(\mathbb{B}_0^1, \mathbb{B}_0^2, \mathbb{B}_0^3)$ and $$\mathbb{B}_1(\vec{\xi}) = \text{diag}(\mathbb{B}_1^1(\xi_1), \mathbb{B}_1^2(\xi_2), \mathbb{B}_1^3(\xi_3)),$$ where $\vec{\xi} = [\xi_1, \xi_2, \xi_3]^\top$.

Write the full physical state of the network as
$$\mathbf{z}(t,\vec{\xi}) = [p^1(t,\xi_1), q^1(t,\xi_1), p^2(t,\xi_2), q^2(t,\xi_2), p^3(t,\xi_3), q^3(t,\xi_3)]^\top,$$
collecting pressure and flux on all three pipes. Subtracting the interpolated boundary data produces the homogenized state
\begin{equation}\label{newvar}
    \mathbf{y}(t,\vec{\xi}) \coloneqq \mathbf{z}(t,\vec{\xi}) - \mathbb{B}_1(\vec{\xi})\mathbf{b}(t).
\end{equation}
By construction, $\mathbf{y}$ satisfies homogeneous boundary conditions at the actuated ends, exactly as the abstract theory requires, while all the boundary information is now carried by the explicit correction term $\mathbb{B}_1(\vec{\xi})\mathbf{b}(t)$.

The natural state space is $Y = [L^2(0,1) \times L^2(0,1)]^3$, square-integrable pressure and flux on each pipe, equipped with the physically motivated energy norm
\begin{equation}
    \|\mathbf{y}\|_{Y}^2 = \sum_{k=1}^3 \int_0^1 \left[ |\tilde{p}^k(\xi_k)|^2 + c^2 |\tilde{q}^k(\xi_k)|^2 \right] d\xi_k,
\end{equation}
which, up to the factor $c^2$, measures the total acoustic energy stored in the pipe network. The dynamics is generated by the operator $A = \text{diag}(A_1, A_2, A_3)$, $A_k = \begin{bmatrix} 0 & -c^2\partial_x \\ -\partial_x & 0 \end{bmatrix}$, acting on functions that satisfy both the homogeneous boundary conditions at the actuated ends and the coupling conditions \eqref{pressure_cont}--\eqref{flow_cont} at the internal junction:
\begin{equation}\label{compactsmooth}
    \mathcal{D}(A) \coloneqq \left\{\mathbf{y} \in [H^1(0,1) \times H^1(0,1)]^3 : 
    \begin{matrix}
        \mathbf{y} = (\tilde{p}^1, \tilde{q}^1, \tilde{p}^2, \tilde{q}^2, \tilde{p}^3, \tilde{q}^3) \\
        \tilde{p}^1(0) = \tilde{q}^2(1) = \tilde{q}^3(1) = 0, \\
        \tilde{p}^1(1) = \tilde{p}^2(0) = \tilde{p}^3(0), \\
        \tilde{q}^1(1) - \tilde{q}^2(0) - \tilde{q}^3(0) = 0.
    \end{matrix}\right\}
\end{equation}
This operator is closed and densely defined and, as is standard for hyperbolic network models of this type, it generates a $C_0$-semigroup of contractions on $Y$: the network dynamics is well posed, and it does not create energy out of nothing.

\subsection{The strategy-to-state maps and game structure}

With the state equation in place, we can describe the game itself. There are three players: Player 1 is the network operator, who sets the entry pressure $u_1$; Players 2 and 3 are suppliers competing in two downstream markets, controlling the exit fluxes $u_2$ and $u_3$ respectively.

This is where the players' individual perspectives enter. Each player $i$ may rely on its own estimate of the friction and inclination coefficients on each pipe, $f_{1,i}^k, f_{2,i}^k$, reflecting for instance different engineering models or private information about pipe conditions. From player $i$'s point of view, the homogenized network dynamics therefore reads
\begin{equation}\label{abs3}
    \begin{cases}
        \mathbf{y}_t(t) = A \mathbf{y}(t) + P_i\mathbf{y}(t) + B_i \mathbf{u}(t), \\
        \mathbf{y}(0) = \mathbf{y}_0^i \coloneqq \mathbf{z}_0^i - \mathbb{B}_1(\vec{\xi})M\mathbf{u}(0),
    \end{cases}
\end{equation}
where $P_i = \text{diag}(P_{1,i}, P_{2,i}, P_{3,i})$, with $P_{k,i} = \begin{bmatrix} 0 & 0 \\ f_{1,i}^k & f_{2,i}^k \end{bmatrix}$, encodes player $i$'s own friction and inclination model, and $\mathcal{B}_i = \mathbb{B}_0 M - P_i \mathbb{B}_1 M - \partial_t \mathbb{B}_1 M$ collects the forcing terms generated by the lifting procedure of the previous subsection. It is precisely this player-dependent operator $P_i$ that keeps the game from collapsing into a single joint optimization problem: no two players are, in general, optimizing against the same physical model of the network.

We work with sufficiently smooth data, so that $\mathbf{y}_0^i \in \mathcal{D}(A)$, and restrict the admissible controls to
$$\mathbb{U} = \left\{\mathbf{u} \in [H^2(0,T)]^3 : M\mathbf{u}(0) = \mathbf{b}_0\right\}, \qquad \mathbf{b}_0 = \begin{bmatrix}
    p_0^1(0) & 0 & 0 & q_0^2 & 0 & q_0^3(0)
\end{bmatrix}^\top,$$
where $\mathbf{b}_0$ simply records the initial boundary values consistent with the initial data. Under these conditions, Theorem \ref{wpp_classic} guarantees that each player's dynamics \eqref{abs3} has a unique classical solution $\mathbf{y}_i \in C([0,T];\mathcal{D}(A)) \cap C^1([0,T];Y)$.

Undoing the change of variables \eqref{newvar} turns the homogenized solution back into physical pressure and flux profiles. This defines player $i$'s strategy-to-state map $\mathcal{Z}_i: \mathbb{U} \to C(\overline{Q})^6$, with $Q = (0,T)\times(0,1)$, as
\begin{equation}
    \mathcal{Z}_i(\mathbf{u}) = \mathcal{S}_i(\mathbf{u}) + \mathbb{B}_1(\vec{\xi})M\mathbf{u},
\end{equation}
where $\mathcal{S}_i$ denotes the solution map of \eqref{abs3}. A standard compact embedding, $\mathcal{D}(A) \xhookrightarrow{c} [H^s(0,1)]^6 \hookrightarrow [C(0,1)]^6$ for $s > 1/2$, combined with the Aubin-Lions lemma, shows that these solution maps have exactly the compactness needed for Assumption \ref{min_constraint}, with $W = [C(\overline{Q})]^6$.

For the gas network to operate safely, pressure and flux must stay within prescribed physical bounds at every point of every pipe and at every time:
\begin{align}\label{cc1}
    0 < p_{\min} \leqslant p^k(t,x) \leqslant p_{\max}, \qquad &\forall (t,x) \in \overline{Q}, \; k=1,2,3, \\ \label{cc2}
    0 \leqslant q^k(t,x) \leqslant q_{\max}, \qquad &\forall (t,x) \in \overline{Q}, \; k=1,2,3.
\end{align}
These pointwise constraints define a closed, convex set $\mathcal{K} \subset W$, and player $i$'s feasible strategies are those consistent with them given the rivals' choices:
\begin{equation}
    \mathbf{U}_i(u_{-i}) = \{u_i \in \mathbb{U}_i^\text{ad} : \mathcal{Z}_i(u) \in \mathcal{K}\}.
\end{equation}
Because $\mathcal{Z}_i$ is affine and $\mathcal{K}$ is convex, Proposition \ref{graph_convex} applies, and the resulting constraint map is graph-convex, the key structural property that makes the existence theory of the previous section available here.

The economic side of the model enters through the market prices faced by suppliers 2 and 3. We use a linear inverse demand function at each exit node,
\begin{equation}
\label{inv_demand}
\pi_i(t)=\alpha_i(t)-\beta_i q^i(t,1)
-\sum_{\substack{j\in\{2,3\}\\ j\neq i}}\sigma_{ij}q^j(t,1) = \alpha_i(t) - \beta_iu_i(t) - \sum_{\substack{j\in\{2,3\}\\ j\neq i}}\sigma_{ij}u_j(t),
\end{equation}
where $\alpha_i(t)>0$ is the market's baseline willingness to pay, $\beta_i>0$ measures how much supplier $i$'s own sales depress its price, and $\sigma_{ij}\ge0$ captures how the rival supplier's sales affect $i$'s price, i.e., the competition between the two downstream markets.

Writing $R_p^k$ and $R_q^k$ for the operators that read off, respectively, the pressure and the flux component on pipe $k$ from the full state vector, the three players' cost functionals are
\begin{align}
    \mathcal{J}_1(\mathbf{u})\! &= \!\!\!\int_0^T\!\!\!\! e^{-\rho t}\! \left[ -\theta [R_q^1 \mathcal{Z}_1(\mathbf{u})](t,0) - \frac{\omega}{2}|u_1(t) - p_{\text{set}}|^2 - \frac{\gamma_1}{2}\big|[R_p^1 \mathcal{Z}_1(\mathbf{u})](t,1) - p_{\text{ref}}\big|^2 \right] dt, \\
    \mathcal{J}_i(\mathbf{u}) &= \int_0^T e^{-\rho t} \left[ -\pi_i(t)u_i(t) + \frac{\gamma_i}{2}\big|[R_p^i \mathcal{Z}_i(\mathbf{u})](t,0) - p_{\text{ref}}\big|^2 \right] dt, \quad i \in \{2,3\}.
\end{align}
Player 1's cost combines a revenue-like term tied to the flux delivered at the inlet, a quadratic penalty for the entry pressure straying from a target $p_\text{set}$, and a penalty for the pressure delivered downstream deviating from a reference value $p_\text{ref}$. Each supplier $i \in \{2,3\}$, meanwhile, minimizes the negative of its own market revenue $\pi_i(t) u_i(t)$ together with a penalty for pressure deviations at its own exit node; all costs are discounted over time at rate $\rho$.

Every hypothesis of Theorem \ref{exist_gen_N} is satisfied by this construction, so the game admits at least one generalized Nash equilibrium $\mathbf{u}^* \in \mathbb{U}$.

\subsection{First-order optimality conditions for Player 1}

We now use the abstract optimality theory of the previous section to turn this existence result into something explicit: first-order conditions describing the equilibrium. We start with Player 1, the network operator. Fix the rivals' equilibrium strategies $u_{-1}^* = (u_2^*, u_3^*)$, and assume, as in Definition \ref{cont_point}, that $\mathbf{u}^*$ is a continuity point for Player 1. Then there exists a Lagrange multiplier, here a vector-valued measure $\mu_1 = (\mu_{p,+}^k, \mu_{p,-}^k, \mu_{q,+}^k, \mu_{q,-}^k)_{k=1}^3 \in [\mathcal{M}(\overline{Q})]^{12}$, associated with the pointwise state constraints \eqref{cc1}--\eqref{cc2}. Complementary slackness means this measure can only be active where the corresponding constraint is active: it "pushes back" only at those times and locations where pressure or flux sits exactly at its upper or lower bound.

Because the map $u_1 \mapsto \mathcal{Z}_1(\mathbf{u})$ is affine, its derivative in a direction $h$ splits cleanly into two pieces, $\mathcal{Z}_1' = \mathcal{Y}_1^0 + L_1$. The first, $\mathcal{Y}_1^0$, is the derivative of the genuine dynamic solution map; the second, $L_1 h = [\mathbb{B}_1^1(\xi_1) [h,0]^\top, \mathbf{0}, \mathbf{0}]^\top$, is the purely algebraic contribution coming from the lifting term in \eqref{newvar}. The resulting stationarity condition reads
\begin{equation}\label{abstract_opt_split}
    0 = \partial_1 \mathcal{J}_1(\mathbf{u}^*) + \lambda_1 + [\mathcal{Y}_1^0]^* \mu_1 + L_1^* \mu_1 \quad \text{in } \mathbb{U}_1^*,
\end{equation}
where $\lambda_1 \in \mathcal{N}_{\mathbb{U}_1^\text{ad}}(u_1^*)$ accounts for the possibility that $u_1^*$ sits on the boundary of its own admissible set $\mathbb{U}_1^\text{ad}$.

To make the term $[\mathcal{Y}_1^0]^*\mu_1$ explicit we introduce an adjoint, or costate, system: a standard device that turns a condition on the derivative of the state into a more tractable condition on a companion variable. Define the adjoint vector $\boldsymbol{\Phi}_1 = [(\varphi^1_1, \psi^1_1), (\varphi^2_1, \psi^2_1), (\varphi^3_1, \psi^3_1)]^\top$ as the unique transposition solution of the backward-in-time hyperbolic system
\begin{equation}\label{adjoint_PDE}
    \begin{cases}
        -(\varphi_1^k)_t - (\psi_1^k)_x - f_{1,1}^k \psi_1^k = \mu_{1,p,+}^k - \mu_{1,p,-}^k, & \text{in } (0,T)\times(0,1), \\[1mm]
        -(\psi_1^k)_t - c^2 (\varphi_1^k)_x - f_{2,1}^k \psi_1^k = \mu_{1,q,+}^k - \mu_{1,q,-}^k, & \text{in } (0,T)\times(0,1), \\[1mm]
        \varphi_1^k(T,x) = 0, \quad \psi_1^k(T,x) = 0, & \text{in } (0,1),
    \end{cases}
\end{equation}
which runs backward from a zero terminal condition at $t=T$, driven by $\mu_1$ wherever a state constraint is active. At the internal junction $v_2$, the adjoint variables inherit their own transmission conditions, mirroring \eqref{pressure_cont}--\eqref{flow_cont}:
\begin{equation}
    \psi_1^1(t,1) = \psi_1^2(t,0) = \psi_1^3(t,0), \qquad \varphi_1^1(t,1) - \varphi_1^2(t,0) - \varphi_1^3(t,0) = 0.
\end{equation}
At the exterior boundary nodes, the pressure-tracking term in Player 1's cost functional feeds directly into the adjoint boundary data, forcing
\begin{equation}
    \psi_1^2(t,1) = 0, \quad \psi_1^3(t,1) = 0, \quad \text{and} \quad c^2\varphi_1^1(t,1) = \gamma_1 e^{-\rho t} \left(p^1(t,1) - p_{\text{ref}}\right).
\end{equation}
Testing the adjoint system against a variation $h \in \mathbb{U}_1$ and integrating by parts, the dynamic part of the dual pairing collapses to a boundary trace:
\begin{equation}
    \langle \mu_1, \mathcal{Y}_1^0 h \rangle = \int_0^T \left( -\psi_1^1(t,0) - \theta e^{-\rho t} \right) h(t) \, dt.
\end{equation}
The algebraic part coming from the lifting operator contributes, in turn,
\begin{equation}
    \langle \mu_1, L_1 h \rangle = \int_0^T \left( \int_0^1 (1-\xi_1) \, d(\mu_{1,p,+}^1 - \mu_{1,p,-}^1)(t, \xi_1) \right) h(t) \, dt.
\end{equation}
Putting these two pieces back into \eqref{abstract_opt_split} gives an explicit, checkable characterization of Player 1's equilibrium strategy: for every admissible test control $v_1 \in \mathbb{U}_1^\text{ad}$, the optimal entry pressure $u_1^*$ satisfies the variational inequality
\begin{equation}\label{final_vi_p1}
\begin{aligned}
    &\int_0^T \left[ -\psi_1^1(t,0) - \theta e^{-\rho t} - \omega e^{-\rho t}(u_1^*(t) - p_{\text{set}})\right] (v_1(t) - u_1^*(t)) \, dt  \\ & \hspace{2cm} +\int_0^T \left[\int_0^1 (1-\xi_1) d\nu_{1,p}^1(t,\xi_1) \right] (v_1(t) - u_1^*(t)) \, dt \geqslant 0,
\end{aligned}
\end{equation}
where $\nu_{1,p}^1 = \mu_{1,p,+}^1 - \mu_{1,p,-}^1 \in \mathcal{M}(\overline{Q})$ is the net multiplier associated with the pressure constraint on pipe 1. In words, \eqref{final_vi_p1} says that at equilibrium no admissible perturbation of $u_1^*$ can lower Player 1's cost once its own marginal cost, the network's feedback carried by the adjoint variable $\psi_1^1(t,0)$, and the shadow price of any active pressure constraint are all accounted for.

\subsection{First-order optimality conditions for Players 2 and 3}

The same procedure applies, with minor changes, to the two suppliers. Throughout, $i \in \{2,3\}$ denotes the supplier under consideration and $j \in \{2,3\}$, $j \neq i$, its rival.

Assuming again that $\mathbf{u}^*$ is a continuity point for player $i$, there is an associated multiplier measure $\mu_i = (\mu_{i,p,+}^k, \mu_{i,p,-}^k, \mu_{i,q,+}^k, \mu_{i,q,-}^k)_{k=1}^3 \in [\mathcal{M}(\overline{Q})]^{12}$ enforcing complementary slackness for player $i$'s state constraints. The only structural difference from Player 1 is which pipe carries the control: edge $2$ for Player 2, edge $3$ for Player 3.

As before, the derivative of the strategy-to-state map splits as $\mathcal{Z}_i' = \mathcal{Y}_i^0 + L_i$, with the algebraic piece now landing on the corresponding exit pipe: $$L_2 h = [\mathbf{0}, \mathbb{B}_1^2(\xi_2) [0,h]^\top, \mathbf{0}]^\top$$ for Player 2, and $L_3 h = [\mathbf{0}, \mathbf{0}, \mathbb{B}_1^3(\xi_3) [0,h]^\top]^\top$ for Player 3.

Supplier $i$'s first-order stationarity condition is
\begin{equation}\label{abstract_opt_split_suppliers}
    0 = \partial_i \mathcal{J}_i(\mathbf{u}^*) + \lambda_i + [\mathcal{Y}_i^0]^* \mu_i + L_i^* \mu_i \quad \text{in } \mathbb{U}_i^*,
\end{equation}
with $\lambda_i \in \mathcal{N}_{\mathbb{U}_i^\text{ad}}(u_i^*)$ again accounting for possible active constraints in $u_i$'s own admissible set.

To make $[\mathcal{Y}_i^0]^* \mu_i$ explicit we introduce, exactly as for Player 1, a player-specific adjoint system $\boldsymbol{\Phi}_i = [(\varphi^1_i, \psi^1_i), (\varphi^2_i, \psi^2_i), (\varphi^3_i, \psi^3_i)]^\top$, now built with player $i$'s own friction/inclination model $P_i$:
\begin{equation}\label{adjoint_PDE_suppliers}
    \begin{cases}
        -(\varphi_i^k)_t - (\psi_i^k)_x - f_{1,i}^k \psi_i^k = \mu_{i,p,+}^k - \mu_{i,p,-}^k, & \text{in } (0,T)\times(0,1), \\[1mm]
        -(\psi_i^k)_t - c^2 (\varphi_i^k)_x - f_{2,i}^k \psi_i^k = \mu_{i,q,+}^k - \mu_{i,q,-}^k, & \text{in } (0,T)\times(0,1), \\[1mm]
        \varphi_i^k(T,x) = 0, \quad \psi_i^k(T,x) = 0, & \text{in } (0,1).
    \end{cases}
\end{equation}
At the junction $v_2$, the adjoint variables again satisfy the same continuity structure as before:
\begin{equation}
    \psi_i^1(t,1) = \psi_i^2(t,0) = \psi_i^3(t,0), \qquad \varphi_i^1(t,1) - \varphi_i^2(t,0) - \varphi_i^3(t,0) = 0.
\end{equation}
One genuine difference from Player 1 appears here: supplier $i$'s cost penalizes the pressure at $v_2$, which sits at the coordinate $x=0$ of both exit pipes 2 and 3. As a result, the tracking term in the objective enters the adjoint equations not at a boundary but at this shared internal point:
\begin{equation}
    c^2 \varphi_i^2(t,0) - c^2 \varphi_i^3(t,0) = \gamma_i e^{-\rho t}\left( p^i(t,0) - p_{\text{ref}} \right).
\end{equation}
The remaining boundary conditions for $\boldsymbol{\Phi}_i$ are homogeneous, since nothing in player $i$'s cost acts there: $\psi_i^1(t,0) = 0$ at the network inlet, and $\psi_i^j(t,1) = 0$ at the rival supplier's exit.

Testing against a variation $h \in \mathbb{U}_i$, the dynamic part of the dual pairing reduces to the adjoint flux trace at player $i$'s own exit:
\begin{equation}
    \langle \mu_i, \mathcal{Y}_i^0 h \rangle = \int_0^T c^2 \varphi_i^i(t,1) h(t) \, dt.
\end{equation}
The algebraic lifting contribution, meanwhile, is spread out along pipe $i$ according to
\begin{equation}
    \langle \mu_i, L_i h \rangle = \int_0^T \left( \int_0^1 \xi_i \, d(\mu_{i,q,+}^i - \mu_{i,q,-}^i)(t, \xi_i) \right) h(t) \, dt.
\end{equation}
Finally, differentiating the cost functional itself brings in the market interaction directly, through the inverse demand relation \eqref{inv_demand}:
\begin{equation}
    \partial_i \mathcal{J}_i(\mathbf{u}^*)h = \int_0^T e^{-\rho t} \left[ -\alpha_i(t) + 2\beta_i u_i^*(t) + \sigma_{ij} u_j^*(t) \right] h(t) \, dt.
\end{equation}

Assembling all these pieces in \eqref{abstract_opt_split_suppliers} gives the explicit equilibrium characterization for supplier $i \in \{2,3\}$: for every admissible test strategy $v_i \in \mathbb{U}_i^\text{ad}$,
\begin{equation}\label{final_vi_suppliers}
\begin{aligned}
    &\int_0^T \left[ c^2 \varphi_i^i(t,1) + e^{-\rho t} \left( -\alpha_i(t) + 2\beta_i u_i^*(t) + \sigma_{ij} u_j^*(t) \right)\right] (v_i(t) - u_i^*(t)) \, dt \\ & \hspace{4cm} + \int_0^T \left[ \int_0^1 \xi_i d\nu_{i,q}^i(t,\xi_i) \right] (v_i(t) - u_i^*(t)) \, dt \geqslant 0,
    \end{aligned}
\end{equation}
where $\nu_{i,q}^i = \mu_{i,q,+}^i - \mu_{i,q,-}^i \in \mathcal{M}(\overline{Q})$ is the net multiplier for the flux constraint on pipe $i$. As for Player 1, this says that at equilibrium supplier $i$ balances the network's feedback, carried by the adjoint variable $\varphi_i^i(t,1)$, against its own marginal market revenue and the shadow price of any active flux constraint along its pipe.

\bibliographystyle{siamplain}
\bibliography{references_p2}

@article {MR1079985,
    AUTHOR = {Ambrosio, Luigi},
     TITLE = {Metric space valued functions of bounded variation},
   JOURNAL = {Ann. Scuola Norm. Sup. Pisa Cl. Sci. (4)},
  FJOURNAL = {Annali della Scuola Normale Superiore di Pisa. Classe di
              Scienze. Serie IV},
    VOLUME = {17},
      YEAR = {1990},
    NUMBER = {3},
     PAGES = {439--478},
      ISSN = {0391-173X,2036-2145},
   MRCLASS = {26A45 (46E40)},
  MRNUMBER = {1079985},
MRREVIEWER = {A.\ Precupanu},
       URL = {http://www.numdam.org/item?id=ASNSP_1990_4_17_3_439_0},
}

@book {MR2500068,
    AUTHOR = {Amann, Herbert and Escher, Joachim},
     TITLE = {Analysis. {III}},
      NOTE = {Translated from the 2001 German original by Silvio Levy and
              Matthew Cargo},
 PUBLISHER = {Birkh\"auser Verlag, Basel},
      YEAR = {2009},
     PAGES = {xii+468},
      ISBN = {978-3-7643-7479-2; 3-7643-7479-2},
   MRCLASS = {00-01 (26-01 28-01 58-01)},
  MRNUMBER = {2500068},
       DOI = {10.1007/978-3-7643-7480-8},
       URL = {https://doi.org/10.1007/978-3-7643-7480-8},
}

@article{arrow_existence_1954,
	title = {Existence of an {Equilibrium} for a {Competitive} {Economy}},
	volume = {22},
	issn = {0012-9682},
	url = {https://www.jstor.org/stable/1907353},
	doi = {10.2307/1907353},
	number = {3},
	urldate = {2024-05-21},
	journal = {Econometrica : journal of the Econometric Society},
	author = {Arrow, Kenneth J. and Debreu, Gerard},
	year = {1954},
	pages = {265--290},
}

@book{aubin_set-valued_2009,
	address = {Boston},
	title = {Set-{Valued} {Analysis}},
	copyright = {https://www.springernature.com/gp/researchers/text-and-data-mining},
	isbn = {978-0-8176-4847-3 978-0-8176-4848-0},
	url = {https://link.springer.com/10.1007/978-0-8176-4848-0},
	language = {en},
	urldate = {2025-12-08},
	publisher = {Birkhäuser},
	author = {Aubin, Jean-Pierre and Frankowska, Hélène},
	year = {2009},
	doi = {10.1007/978-0-8176-4848-0},
}

@book{bensoussan_impulse_1987,
	series = {Modern {Applied} {Mathematics} {Series}},
	title = {Impulse {Control} and {Quasi} {Variational} {Inequalities}},
	isbn = {978-0-471-82973-7},
	publisher = {John Wiley \& Sons Canada, Limited},
	author = {Bensoussan, Alain and Lions, Jacques-Louis},
	year = {1987},
}

@misc{bongarti2025structureversusregularitysetvalued,
      title={Structure versus regularity of set-valued maps in convex generalized {N}ash equilibrium problems in {B}anach spaces},
      author={Marcelo Bongarti and Michael Hintermüller},
      year={2025},
      eprint={2512.12831},
      archivePrefix={arXiv},
      primaryClass={math.OC},
      note = {\url{https://arxiv.org/abs/2512.12831}, \doi{10.48550/arXiv.2512.12831}}
}

@article{bongarti_optimal_2024,
	title = {Optimal {Boundary} {Control} of the {Isothermal} {Semilinear} {Euler} {Equation} for {Gas} {Dynamics} on a {Network}},
	volume = {89},
	copyright = {All rights reserved},
	issn = {0095-4616, 1432-0606},
	url = {https://link.springer.com/10.1007/s00245-023-10088-0},
	doi = {10.1007/s00245-023-10088-0},
	language = {en},
	number = {2},
	urldate = {2024-05-22},
	journal = {Applied Mathematics \& Optimization},
	author = {Bongarti, Marcelo and Hintermüller, Michael},
	month = apr,
	year = {2024},
	pages = {36},
}

@inproceedings{bongarti_singular_2020,
	address = {Cham},
	title = {Singular {Thermal} {Relaxation} {Limit} for the {Moore}-{Gibson}-{Thompson} {Equation} {Arising} in {Propagation} of {Acoustic} {Waves}},
	copyright = {All rights reserved},
	isbn = {978-3-030-46079-2},
	doi = {10.1007/978-3-030-46079-2_9},
	language = {en},
	booktitle = {Semigroups of {Operators} – {Theory} and {Applications}},
	publisher = {Springer International Publishing},
	author = {Bongarti, Marcelo and Charoenphon, Sutthirut and Lasiecka, Irena},
	editor = {Banasiak, Jacek and Bobrowski, Adam and Lachowicz, Mirosław and Tomilov, Yuri},
	year = {2020},
	pages = {147--182},
}

@article{chan_generalized_1982,
	title = {The {Generalized} {Quasi}-{Variational} {Inequality} {Problem}},
	volume = {7},
	issn = {0364-765X, 1526-5471},
	url = {https://pubsonline.informs.org/doi/10.1287/moor.7.2.211},
	doi = {10.1287/moor.7.2.211},
	number = {2},
	urldate = {2024-07-19},
	journal = {Mathematics of Operations Research},
	author = {Chan, Der-San and Pang, Jong-Shi},
	month = may,
	year = {1982},
	pages = {211--222},
}

@incollection{debreu_chapter_1982,
	title = {Chapter 15 {Existence} of competitive equilibrium},
	volume = {2},
	copyright = {https://www.elsevier.com/tdm/userlicense/1.0/},
	isbn = {978-0-444-86127-6},
	url = {https://linkinghub.elsevier.com/retrieve/pii/S1573438282020104},
	language = {en},
	urldate = {2024-07-18},
	booktitle = {Handbook of {Mathematical} {Economics}},
	publisher = {Elsevier},
	author = {Debreu, Gerard},
	year = {1982},
	doi = {10.1016/S1573-4382(82)02010-4},
	pages = {697--743},
}

@misc{dvurechensky2024cournotnashmodelcoupledhydrogen,
	title = {A cournot-nash model for a coupled hydrogen and electricity market},
	url = {https://arxiv.org/abs/2410.20534},
	author = {Dvurechensky, Pavel and Geiersbach, Caroline and Hintermüller, Michael and Kannan, Aswin and Kater, Stefan and Zöttl, Gregor},
	year = {2024},
	note = {arXiv: 2410.20534 [math.OC]},
}

@misc{dvurechensky_cournot-nash_2024,
	title = {A {Cournot}-{Nash} {Model} for a {Coupled} {Hydrogen} and {Electricity} {Market}},
	url = {http://arxiv.org/abs/2410.20534},
	doi = {10.48550/arXiv.2410.20534},
	urldate = {2025-12-05},
	publisher = {arXiv},
	author = {Dvurechensky, Pavel and Geiersbach, Caroline and Hintermüller, Michael and Kannan, Aswin and Kater, Stefan and Zöttl, Gregor},
	month = oct,
	year = {2024},
	note = {arXiv:2410.20534 [math]},
}

@article{hintermueller_pde-constrained_2013,
	title = {A {PDE}-{Constrained} {Generalized} {Nash} {Equilibrium} {Problem} with {Pointwise} {Control} and {State} {Constraints}},
	volume = {9},
	issn = {1348-9151},
	number = {2},
	journal = {Pacific Journal of Optimization},
	author = {Hinterm\"uller, Michael and Surowiec, Thomas},
	month = apr,
	year = {2013},
	pages = {251--273},
}

@article{hintermuller_generalized_2015,
	title = {Generalized {Nash} {Equilibrium} {Problems} in {Banach} {Spaces}: {Theory}, {Nikaido}–{Isoda}-{Based} {Path}-{Following} {Methods}, and {Applications}},
	volume = {25},
	issn = {1052-6234, 1095-7189},
	url = {http://epubs.siam.org/doi/10.1137/130929203},
	doi = {10.1137/130929203},
	number = {3},
	urldate = {2025-07-03},
	journal = {SIAM Journal on Optimization},
	author = {Hintermüller, Michael and Surowiec, Thomas and Kämmler, Alexander},
	month = jan,
	year = {2015},
	pages = {1826--1856},
}

@book{ioffe_theory_2010,
	address = {Amsterdam New York},
	series = {Studies in mathematics and its applications},
	title = {Theory of extremal problems},
	isbn = {978-0-444-85167-3 978-0-08-087527-9},
	language = {eng},
	number = {v. 6},
	publisher = {North-Holland Pub. Co},
	author = {Ioffe, Aleksandr D. and Tichomirov, Vladimir M.},
	note = {Translated from the Russian by Karol Makowski},
	year = {2010},
}

@book{lasiecka_control_2000,
	address = {Cambridge},
	series = {Encyclopedia of {Mathematics} and its {Applications}},
	title = {Control {Theory} for {Partial} {Differential} {Equations}: {Continuous} and {Approximation} {Theories}: {Volume} 1: {Abstract} {Parabolic} {Systems}},
	volume = {1},
	isbn = {978-0-521-43408-9},
	shorttitle = {Control {Theory} for {Partial} {Differential} {Equations}},
	url = {https://www.cambridge.org/core/books/control-theory-for-partial-differential-equations/3D51EEA85884F240F6FC5EEA024B8460},
	urldate = {2025-12-05},
	publisher = {Cambridge University Press},
	author = {Lasiecka, Irena and Triggiani, Roberto},
	year = {2000},
	doi = {10.1017/CBO9781107340848},
}

@article{lions_hierarchic_1994,
	title = {Hierarchic control},
	volume = {104},
	copyright = {http://www.springer.com/tdm},
	issn = {0253-4142, 0973-7685},
	url = {http://link.springer.com/10.1007/BF02830893},
	doi = {10.1007/BF02830893},
	language = {en},
	number = {1},
	urldate = {2024-07-22},
	journal = {Proceedings Mathematical Sciences},
	author = {Lions, J. L.},
	month = feb,
	year = {1994},
	pages = {295--304},
}

@article{lions_remarks_1994,
	title = {Some remarks on {Stackelberg's} optimization},
    author = {Lions, J.L.},
	volume = {04},
	issn = {0218-2025, 1793-6314},
	url ={https://www.worldscientific.com/doi/10.1142/S0218202594000273},
	doi = {10.1142/S0218202594000273},
	language = {en},
	number = {04},
	urldate = {2024-07-22},
	journal = {Mathematical Models and Methods in Applied Sciences},
	month = aug,
	year = {1994},
	pages = {477--487},
}

@article{machowska_closed-loop_2022,
	title = {Closed-loop {N}ash equilibrium for a partial differential game with application to competitive personalized advertising},
	volume = {140},
	issn = {0005-1098},
	url = {https://www.sciencedirect.com/science/article/pii/S0005109822000656},
	doi = {10.1016/j.automatica.2022.110220},
	urldate = {2025-12-05},
	journal = {Automatica},
	author = {Machowska, Dominika and Nowakowski, Andrzej and Wiszniewska-Matyszkiel, Agnieszka},
	month = jun,
	year = {2022},
	pages = {110220},
}

@article{marchand_abstract_2012,
	title = {An abstract semigroup approach to the third-order {Moore}–{Gibson}–{Thompson} partial differential equation arising in high-intensity ultrasound: structural decomposition, spectral analysis, exponential stability},
	volume = {35},
	copyright = {Copyright © 2012 John Wiley \& Sons, Ltd.},
	issn = {1099-1476},
	shorttitle = {An abstract semigroup approach to the third-order {Moore}–{Gibson}–{Thompson} partial differential equation arising in high-intensity ultrasound},
	url = {https://onlinelibrary.wiley.com/doi/abs/10.1002/mma.1576},
	doi = {10.1002/mma.1576},
	language = {en},
	number = {15},
	urldate = {2025-12-16},
	journal = {Mathematical Methods in the Applied Sciences},
	author = {Marchand, R. and McDevitt, T. and Triggiani, R.},
	year = {2012},
	note = {\_eprint: https://onlinelibrary.wiley.com/doi/pdf/10.1002/mma.1576},
	pages = {1896--1929},
}

@article{nikaido_note_1955,
	title = {Note on non-cooperative convex game},
	volume = {5},
	issn = {0030-8730},
	url = {https://projecteuclid.org/journals/pacific-journal-of-mathematics/volume-5/issue-5/Note-on-non-cooperative-convex-game/pjm/1103045138.full},
	doi = {10.2140/pjm.1955.5.807},
	number = {5},
	urldate = {2024-07-19},
	journal = {Pacific Journal of Mathematics},
	author = {Nikaidô, Hukukane and Isoda, Kazuo},
	month = dec,
	year = {1955},
	pages = {807--815},
}

@article{pang_quasi-variational_2005,
	title = {Quasi-variational inequalities, generalized {Nash} equilibria, and multi-leader-follower games},
	volume = {2},
	issn = {1617-5891, 1617-5905},
	url = {https://doi.org/10.1007/s10287-004-0020-y},
	doi = {10.1007/s10287-004-0020-y},
	number = {1},
	urldate = {2024-05-21},
	journal = {Computational Management Science},
	author = {Pang, Jong-Shi and Fukushima, Masao},
	month = jan,
	year = {2005},
	pages = {21--56},
}

@book{pazy_semigroups_1983,
	address = {New York, NY},
	series = {Applied {Mathematical} {Sciences}},
	title = {Semigroups of {Linear} {Operators} and {Applications} to {Partial} {Differential} {Equations}},
	volume = {44},
	copyright = {http://www.springer.com/tdm},
	isbn = {978-1-4612-5563-5 978-1-4612-5561-1},
	url = {http://link.springer.com/10.1007/978-1-4612-5561-1},
	urldate = {2025-12-07},
	publisher = {Springer},
	author = {Pazy, A.},
	year = {1983},
	doi = {10.1007/978-1-4612-5561-1},
}

@article{roubicek_nash_2007,
	title = {On {Nash} {Equilibria} for {Noncooperative} {Games} {Governed} by the {Burgers} {Equation}},
	volume = {132},
	issn = {0022-3239, 1573-2878},
	url = {https://doi.org/10.1007/s10957-007-9238-y},
	doi = {10.1007/s10957-007-9238-y},
	number = {1},
	urldate = {2024-07-19},
	journal = {Journal of Optimization Theory and Applications},
	author = {Roubíček, Tomáš},
	month = mar,
	year = {2007},
	pages = {41--50},
}

@incollection{roubicek_noncooperative_1999,
	address = {Basel},
	series = {International {Series} of {Numerical} {Mathematics}},
	title = {Noncooperative {Games} with {Elliptic} {Systems}},
	isbn = {978-3-0348-8683-3 978-3-0348-8690-1},
	url = {http://link.springer.com/10.1007/978-3-0348-8690-1_14},
	language = {en},
	urldate = {2024-05-21},
	booktitle = {Optimal {Control} of {Partial} {Differential} {Equations}},
	publisher = {Birkhäuser},
	author = {Roubíček, Tomáš},
	editor = {Hoffmann, Karl-Heinz and Leugering, Günter and Tröltzsch, Fredi and Caesar, Siegfried},
	year = {1999},
	doi = {10.1007/978-3-0348-8690-1_14},
	pages = {245--255},
}

@article{simon_compact_1986,
	title = {Compact sets in the {spaceLp}({O},{T}; {B})},
	volume = {146},
	issn = {1618-1891},
	url = {https://doi.org/10.1007/BF01762360},
	doi = {10.1007/BF01762360},
	language = {en},
	number = {1},
	urldate = {2025-12-07},
	journal = {Annali di Matematica Pura ed Applicata},
	author = {Simon, Jacques},
	month = dec,
	year = {1986},
	pages = {65--96},
}

@article{zame_competitive_1987,
	title = {Competitive {Equilibria} in {Production} {Economies} with an {Infinite}-{Dimensional} {Commodity} {Space}},
	volume = {55},
	issn = {0012-9682},
	url = {https://www.jstor.org/stable/1911262},
	doi = {10.2307/1911262},
	number = {5},
	urldate = {2024-05-22},
	journal = {Econometrica},
	author = {Zame, William R.},
	year = {1987},
	note = {Publisher: [Wiley, Econometric Society]},
	pages = {1075--1108},
}
\end{document}